\documentclass[final]{siamltex}
\usepackage{amsmath,bm}
\usepackage{amssymb}
\usepackage{xcolor}
\usepackage{graphics}
\usepackage{verbatim}
\usepackage{tabularx}
\usepackage{subfig}
\usepackage{algpseudocode}
\usepackage{float}
\usepackage{mathtools}
\usepackage{array,multirow}
\usepackage{pgfplots}
\usepackage{tikz}
\usetikzlibrary{matrix,arrows,decorations,calc,decorations.pathreplacing,backgrounds}
\usetikzlibrary{decorations.pathmorphing,decorations.text}
\usetikzlibrary{scopes,decorations.markings,positioning,shapes.arrows,chains,fit,shapes}

\newtheorem{algorithm}[theorem]{Algorithm}
\newtheorem{remark}[theorem]{Remark}

\def\ocirc#1{\ifmmode\setbox0=\hbox{$#1$}\dimen0=\ht0
    \advance\dimen0 by1pt\rlap{\hbox to\wd0{\hss\raise\dimen0
    \hbox{\hskip.2em$\scriptscriptstyle\circ$}\hss}}#1\else
    {\accent"17 #1}\fi}

\newcommand{\real}{\mathbb{R}}
\newcommand{\bff}{\bm{f}}

\def\sst{\,\arrowvert\,}

\def \bg{\bm{g}}
\def \bx{\bm{x}}
\def \by{\bm{y}}
\def \bw{\bm{w}}
\def \bz{\bm{z}}
\def \bb{\bm{b}}
\def \ba{\bm{a}}
\def \be{\bm{e}}
\def \xhat{\hat{\bm{x}}}
\def \ns{\mathcal{N}}
\def \rng{\mathcal{R}}

\def\rank{\text{rank}}
\def\dv{\text{div}}

\def \dsp{\displaystyle}
\def \bv{\bm{v}}

\def \bu{\bm{u}}

\title{Prestructuring sparse matrices with dense rows and columns via null space methods} 

\author{Jason S.~Howell\thanks{Department of Mathematics, College of Charleston, Charleston,
SC, 29424, USA. email: {\tt howelljs@cofc.edu}}}
\begin{document}
\maketitle

\begin{abstract}
Several applied problems may produce large sparse matrices with a small number of dense rows and/or columns, which can adversely affect the performance of commonly used direct solvers.  By posing the problem as a saddle point system, an unconventional application of a null space method can be employed to eliminate dense rows and columns.  The choice of null space basis is critical in retaining the overall sparse structure of the matrix.   A one-sided application of the null space method is also presented to eliminate either dense rows or columns.  These methods can be considered techniques that modify the nonzero structure of the matrix before employing a direct solver, and may result in improved direct solver performance.  
\end{abstract}

\begin{keywords}null space method, saddle point problem, direct method, dense row, dense column, bordered matrix \end{keywords}

\begin{AMS}65F05, 65N22, 65F50\end{AMS}

\pagestyle{myheadings}
\thispagestyle{plain}
\markboth{JASON HOWELL}{PRESTRUCTURING 
SPARSE MATRICES WITH DENSE ROWS}

\section{Introduction}\label{sec:intro}
Modern direct solvers for large sparse linear systems usually consist of three phases: a symbolic analysis phase, a numerical factorization phase, and a solution phase, with some approaches combining aspects of symbolic and numerical factorization \cite{davisbook}.
Much of the recent advancements in the efficiency of these solvers can be attributed in no small part to advances in algorithms used in the symbolic phase, which examines the {\em structure} of the coefficient matrix, i.e. the pattern of zero/nonzero entries.  A thorough analysis of the structure lends intelligence that helps algorithms  avoid numerical operations on zero entries, order rows and/or columns for a more efficient numerical factorization, reduce the amount of fill (zero entries being converted to nonzeros) in $LU$ factorizations, and results in more efficient data structures and memory allocation.  Thus the structure of the coefficient matrix imparts significant influence on the efficiency in solving a large sparse linear system via direct methods.

This contrasts with iterative methods, the efficiency of which may be influenced greatly by the condition number of the coefficient matrix.  To improve the potential performance of iterative methods, preconditioning techniques have been studied extensively \cite{benziprecon}.  In essence, preconditioning can be thought of as a modification of the linear system prior to solution with the intention of improving iterative solver performance.

A preconditioning analogue for direct solvers would be a technique that modifies the nonzero structure of the coefficient matrix prior to entering the symbolic analysis phase of a modern direct solver, with the intention of realizing gains in solver performance.  This sort of approach can be thought of as a  {\em prestructuring technique}.
One potential pitfall encountered in the symbolic analysis phase of a direct solver is the presence of one or more ``dense'' rows  in an otherwise sparse coefficient matrix.  Here the definition of dense will be taken from \cite{davisbook}: a row of an $n\times n$ matrix is {\em dense} if it contains more than $10\sqrt{n}$ entries.  
The work presented here develops a prestructuring technique that may improve the performance of the direct solver by removing  dense rows (as well as a dense columns) from the structure of the matrix.  

The general approach to this technique will be derived from posing the linear system  as a {\em saddle point} system. 
Several applications of numerical methods result in large, sparse systems of linear equations $M\bu=\bb$ where
\begin{equation}M=
\begin{bmatrix}
A& B_1^T\\B_2&C
\end{bmatrix},\qquad
\bu=
\begin{bmatrix}
\bx\\\by
\end{bmatrix},\qquad  \bb=\begin{bmatrix}
\bff \\ \bg 
\end{bmatrix},
\label{eq:e1}
\end{equation}
where $M\in \real^{(n+m)\times (n+m)}$ is invertible, $n\ge m$, $\dsp A\in \real^{n\times n}$, $B_1, B_2\in \real^{m\times n}$, $C\in \real^{m\times m}$, $\bx, \bff\in\real^{n}$, and $\by, \bg\in \real^m$.  Matrices $M$ with the structure in \eqref{eq:e1} with $B_1\ne B_2$ are referred to as {\em generalized} saddle point matrices.    Saddle point problems arise in several contexts, including computational fluid dynamics and solid mechanics, constrained optimization problems, optimal control, circuit analysis, economics, and finance. 
Consequently, the numerical solution of saddle point problems is a rich field of study -  a comprehensive survey of methods for solving saddle point linear systems is given in \cite{ben051}.  In several of the aforementioned applications, $M$ and/or $A$ often have a  structure that makes modern direct solvers a popular choice.  

Most solution approaches for saddle point problems are designed to exploit the structure of $M$ and properties that $A, B_1$, and $B_2$ might also satisfy. 
Null space methods, also known as reduced Hessian or force methods, constitute one class of solution techniques for saddle point problems when $C=0$.  They have arisen in several contexts, including constrained optimization \cite{colemanbook}, structural and fluid mechanics \cite{plemmons90,arioli2003}, electrical engineering \cite{chua}, and, more recently, mixed finite element approximation of Darcy and Stokes problems \cite{arioli06,leborne071,leborne081, leborne091} and continuum models for liquid crystals \cite{ramage13}.  

Of particular interest in the current context is the situation when $n$ is large, $A$ is sparse, $m$ is very small, $\rank(A)\ge n-m$, and one or both of $B_1, B_2$ contain dense rows.   In this case $M$ may be referred to  as  a ``bordered'' system \cite{govaerts1994,govaertsbook} as the possibly-dense blocks $B_2$ and  $B_1^T$ border the sparse matrix $A$. 
This situation can occur in finite element approximation of PDEs in which a constraint (such as a particular unknown having zero mean over the spatial domain) is enforced on a subset of unknowns via a global scalar Lagrange multiplier.    Systems with this particular structure can also arise in numerical continuation methods for large nonlinear systems of equations which are parametrized with respect to a pseudo-arclength parameter \cite{Keller1977,Keller1978,Keller82,Keller83}.   Other applications that can result in the presence of at least one dense row and/or column in a sparse matrix include optimization, least squares, and circuit analysis \cite{najm2010circuit}.

The general principle behind null space methods is to characterize the null space of the constraint (off-diagonal) operators in a saddle point system and use that characterization to reduce the saddle point system to two significantly smaller systems with nice properties of size $(n-m)\times (n-m)$ and $m\times m$.    Thus, most applications of null space methods are geared towards problems where $n-m$ is small.  Another particular advantage of null space methods is that the matrix $A$ need not be invertible.  
The main drawback of a null space method is the work required to compute null bases for the matrices $B_1$ and $B_2$.  In general this can be a difficult problem \cite{colemanpothen1,colemanpothen2} and even more so when the matrix $Z_k$, whose columns  form a basis for the null space of $B_k$, is required to have additional properties, such as sparsity or orthogonality \cite{gilbertheath87,pothen89}.  The null space method also requires $C=0$.

The objective of this paper is to consider an  unconventional application of a null space method to \eqref{eq:e1} when $m$ is small and at least one of $B_1$ or $B_2$ is dense.  The null space method can  eliminate the  dense row and column from the linear system while retaining as much of the sparse structure in $A$ as possible, thereby preserving properties of $A$ that make it attractive for direct linear solvers.  With the right choice of null bases (of dimension $n-1$) the resulting system can often be solved in a fraction of the time of the original system using direct solvers. 
While a typical application of the null space method is designed to significantly reduce the overall size of the linear systems involved but requires significant overhead in computing the null basis, the application of a null space method described here only reduces the size of the overall system \eqref{eq:e1} by $m+1$ equations and $m+1$ unknowns and the null basis is easy to compute. Therefore, the techniques developed here can be viewed as prestructuring techniques for systems with dense rows or columns before solution with a direct solver.

The method presented here can also be applied when one of $B_1$ or $B_2$ is not dense, and can be applied to matrices with dense rows and/or columns anywhere in the matrix by first performing row or column interchanges to arrive at the structure given in \eqref{eq:e1}.  A ``one-sided'' application of the null space method to systems with $C\ne 0$ will also be developed which, when employed to remove dense rows, can  increase solver performance.

It should be noted that several numerical algorithms for bordered systems arising in continuation methods and other applications have been studied  \cite{chan84,chan84b,chan84d,chansaad85,chan86,
govaerts1991,govaerts1994,calvetti2000} and may be applied to \eqref{eq:e1} when $B_1$ or $B_2$ is dense.  Several of these bordering/block-elimination/deflation methods require finding bases of the left and right null spaces of $A$, or require $C\ne 0$ in \eqref{eq:e1} and a (potentially dense) rank-1 update of $A$.   Other techniques for dealing with dense rows that have arisen in the study of solving linear least squares problems  include splitting algorithms
\cite{sun1995dealing} as well as
matrix stretching \cite{adlers2000,grcar2012}, which increases the size of the coefficient matrix to modify the sparsity pattern, thereby reducing the effect of dense rows.   
However, no existing methods for dealing with dense rows in sparse matrices are known to maintain or reduce the overall size of the problem and retain the overall sparsity and structural properties (such as bandedness) of the remainder of the coefficient matrix.

The rest of this paper is organized as follows.  In Section \ref{sec:graph}, the mathematical background is described, and an example of how a dense row can affect certain aspects of direct solver performance are explored.  In Section \ref{sec:ns}, the general null space method for reducing the system \eqref{eq:e1} is described, as well as a one-sided null space method.  In Section \ref{sec:alg}, the particular approach for constructing the null space basis is given.  Several numerical examples of the method applied to particular problems are presented Section \ref{sec:app}, which demonstrate the utility of the method to varying degrees, and a brief summary is given in Section \ref{sec:sum}.

\section{Background and Motivation}\label{sec:graph}

\subsection{Notation and Definitions}

Throughout this paper, matrices and matrix blocks will be represented  with capital letters (including blocks that have only one row or column), and vectors will be represented as block matrices or with boldface lowercase letters.  The matrix $A$ is comprised of entries $(a_{ij})$, and the $i$th row of $A$ is given by $A_{i*}$.  The $j$th column of $A$ is denoted by $A_{*j}$ or $\ba_j$.  The Euclidean norm of the vector $\bx$ is $\|\bx\|$ and the infinity norm is denoted $\|\bx\|_\infty$.  The number of nonzero entries in a matrix or vector will be denoted by $|\cdot |$.  The rest of the paper will ignore {\em numerical cancellation}, that is, the conversion of nonzero entries to zero entries via numerical operations.

 For a matrix $A\in\real^{n\times m}$, the range (column space) of $A$ is
\[\rng(A)=\{\bv\in\real^n\sst \bv=A\bx \text{ for some } \bx\in \real^m\},\]  and the null space of $A$ is denoted by
\[\ns(A) = \{\bv\in\real^m \sst A\bv=0\}.\]
When the set of vectors in the range or null space come from  a particular subspace $\mathcal{V}$ of $\real^n$ or $\real^m$, these will be denoted by $\rng(A|\mathcal{V})$ or $\ns(A|\mathcal{V})$, respectively.

An $m\times n$ matrix $A$ with $m\ge n$ has the {\em Hall property} if every set of $k$ columns, $1\le k \le n$, contains nonzeros in at least $k$ rows, and $A$ has the {\em strong Hall property} when every set of $k$ columns, $1\le k \le n-1$, contains nonzeros in at least $k+1$ rows.

Graphs are defined by vertex and edge sets.  Vertices will be denoted by number;  $\langle i,j\rangle$ represents a directed or undirected edge between vertices $i$ and $j$, in which case vertices $i$ and $j$ are {\em adjacent}.  The {\em degree} of a vertex $i$ is the number of adjacent vertices, which is also the number of edges that originate or terminate at vertex $i$.  The {\em indegree} and {\em outdegree} are the number of edges leading into and out of a vertex of a directed graph, respectively.

It will be useful to describe the nonzero pattern of a matrix using a bipartite graph: for a general $m\times n$ matrix $A$, the bipartite graph $H(A)$ consists of row vertices $1',2',\ldots,m'$, column vertices $1,2,\ldots,n$, and edges $\langle i',j\rangle$ when $A_{ij}\ne 0$.  In this paper the convention will be to display row vertices along a horizontal line below the horizontal line containing the column vertices.  Additionally, with bipartite graphs, the direction of edges are obvious so arrows will be suppressed.

The {\em layered graph} $L(AB)$ of a matrix product $AB$ is a graph with three rows of vertices: it is the graph $H(B)$ placed on top of the graph $H(A)$ so that the column vertices of $A$ are the row vertices of $B$.  The layered graph is easily generalized to represent the product $A_1A_2\cdots A_r$.  There is a {\em path} from vertex $i$ to vertex $j$ in a directed (or undirected) graph if there is a sequence of vertices $(v_1,v_2,\ldots,v_k)$ such that $v_1=i$, $v_k=j$, and $v_\ell$ is adjacent to $v_{\ell+1}$ for $1\le \ell <k$.

To illustrate how the nonzero pattern of the product of sparse matrices is determined from the associated layered graph, the following result, adapted from Proposition 1 of \cite{cohen96}, will be utilized.

\begin{lemma}\label{lem:cohen}
Let $A\in \real^{m\times n}$ and $B\in\real^{n\times p}$.  Assuming no numerical cancellation,
\begin{enumerate}
\item The number of nonzero entries in the $i$th row of $AB$ equals the number of $B$ column vertices in the layered graph $L(AB)$ of $AB$ that the $i$th row vertex of $A$ reaches via directed paths.
\item The number of nonzero entries in the $i$th column of $AB$ is equal to the number of $A$ row vertices in the layered graph $L(AB)$ of $AB$ that can reach the $i$th column vertex of $B$.
\end{enumerate} 
\end{lemma}

An example of the application of Lemma \ref{lem:cohen} is illustrated below.  In \eqref{eq:ab} the nonzero entries of each of $A$ and $B$ are represented with $\times$, and Figure \ref{fig:lab} gives the layered graph $L(AB)$.  Each possible path from a row vertex of $A$ on the bottom row to a column vertex of $B$ on the top row gives a nonzero entry in $C=AB$, provided there is no numerical cancellation.

\begin{equation}
AB=\begin{bmatrix}
&\times&&\times&\\
\times&&&\times&\times\\
&&\times&&\\
\times&\times&&&\times\\
&&&\times&
\end{bmatrix}\begin{bmatrix}
\times&&&&\\
&\times&&\times&\\
&\times&&&\times\\
&&\times&&\\
\times&\times&&&
\end{bmatrix}=\begin{bmatrix}
&\times&&\times&\\
\times&\times&\times&&\\
&\times&&&\times\\
\times&\times&&\times&\\
&&\times&&
\end{bmatrix}=C.
\label{eq:ab}
\end{equation}
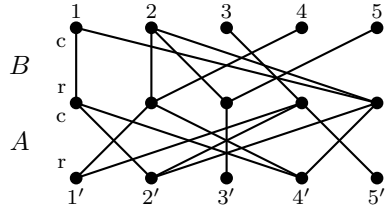
\begin{figure}[H]
\centering
\begin{tikzpicture}[style=thick,scale=1.0]
\draw (1,0)--(2,1);\draw (1,0)--(4,1);
\draw (2,0)--(1,1);\draw (2,0)--(4,1);\draw (2,0)--(5,1);
\draw (3,0)--(3,1);
\draw (4,0)--(1,1);\draw (4,0)--(2,1);\draw (4,0)--(5,1);
\draw (5,0)--(4,1);
\draw (1,1)--(1,2);
\draw (2,1)--(2,2);\draw (2,1)--(4,2);
\draw (3,1)--(2,2);\draw (3,1)--(5,2);
\draw (4,1)--(3,2);
\draw (5,1)--(1,2);\draw (5,1)--(2,2);
\foreach \x in {1,2,3,4,5}{
\filldraw (\x,0) circle (2pt)node[below]{\footnotesize {$\x'$}};
\filldraw (\x,2) circle (2pt) node[above]{\footnotesize {$\x$}};}
\foreach \x in {1,2,3,4,5}{
\filldraw (\x,1) circle (2pt);}
\node at(0.25,0.5){ $A$};
\node at(0.25,1.5){$B$};
\node[above left] at (1,0) {\footnotesize r};
\node[above left] at (1,1) {\footnotesize r};
\node[below left] at (1,1) {\footnotesize c};
\node[below left] at (1,2) {\footnotesize c};
\end{tikzpicture}
\caption{Illustration of the nonzero pattern of $C=AB$ in \eqref{eq:ab} using the layered graph $L(AB)$.}
\label{fig:lab}
\end{figure}

\subsection{Dense Rows/Columns and Direct Solvers}

  While a complete analysis of the effect that a dense row and/or column has on all aspects of particular direct solvers is beyond the scope of this work, a brief discussion and some examples of easily observed phenomena will be presented.  
  
 During symbolic analysis, solvers typically employ directed and/or undirected graph data structures and algorithms for analyzing the nonzero pattern of a sparse matrix.  This in turn is often used to allocate memory to store the $L$ and $U$ factors of the matrix, as well as determine row or column orderings that may reduce the amount of fill encountered during the numerical factorization.  The reference \cite{davisbook} provides a broad summary of data structures and graph algorithms for direct Cholesky, QR, and LU factorization procedures.  
Summarizing several published theoretical results \cite{gilbert94,georgeng85,gilbertng}, Theorems 6.1 and 6.2 of \cite{davisbook} specify the following about the strong Hall matrix $A$, where $PA=LU$ (with $P$ determined by partial pivoting) and $A=QR$ (the QR factorization), assuming no numerical cancellation:
\begin{enumerate}
\item An upper bound for the structure of $U$ is given by the structure of $R$.
\item An upper bound for the structure of $L$ is given by the structure of $V$, the lower trapezoidal matrix of Householder vectors used in the QR factorization of $A$.
\end{enumerate}

An {\em elimination tree} \cite{schreiber82} is a data structure that is used in many aspects of direct solvers, including storage, row and column ordering, and symbolic and numeric factorization.
For a matrix $M$ satisfying the strong Hall property, the {\em column elimination tree}
\cite{eisenstatliu05,pothentoledo} of $M$ is the elimination tree of $M^TM$, and it  can be used to predict potential intercolumn dependencies.  Recent work has pointed to the {\em row merge tree} \cite{oliveira2000,grigori07,grigori09}, based on the {\em row merge matrix} \cite{georgeng87}, as an alternative for structure prediction and can be employed for matrices satisfying only the Hall property.  These trees and related structures  are employed to develop upper bounds on the number of nonzeros in $Q, R, L$, or $U$ by various means.  

However, the presence of a dense row in $M$ could lead to significant, if not catastrophic, fill during the elimination process.  In fact, the presence of a full row in a matrix $M$ implies that $M^TM$ is completely full, which results in a column elimination tree of full height.  Additionally, the row merge matrix of $M$ will also be full.  Together these may lead to overestimates of column dependencies and the number of nonzeros in the $L$ and $U$ factors of $M$ and adversely affect several aspects of solver performance.
Additionally, dense columns may affect algorithms that are used to determine row and column orderings that minimize fill \cite{davisumfpack,davisbook,davisduff,daviscolamd}, especially when $A$ is not SPD or has some zero diagonal entries.

\subsection{Example of Symbolic Analysis with a Dense Row and Column}\label{sec:sym2}
The MATLAB computing environment, which utilizes the UMFPACK solver \cite{davisumfpack} for sparse, square, nonsymmetric matrices, provides a symbolic factorization command ({\tt symbfact(M,'col')}) that returns upper-bound estimates on the number of nonzeros in the numerical factorization of a matrix, as well as the column elimination tree.  Additionally, MATLAB provides a {\tt spparms} setting that can display detailed information about the sparse matrix algorithms employed when the MATLAB linear solve (\verb|\|) is executed.

Assume $m=1$.  Here let $M$ be a (lower-right pointing) {\em arrowhead matrix}, which for $m=1$ consists of a diagonal $A$, dense rows $B_1$ and $B_2$, and nonzero $C$.  Arrowhead matrices, which satisfy the Hall property, arise in several applications \cite{oleary90}, and one solution approach consists of transforming arrowhead matrices to tridiagonal matrices via chasing algorithms \cite{zha92,oliveira98}.  
Set $A=I_n$ in \eqref{eq:e1} and $B_2$ to be a random $n$-vector with entries between $0$ and $1$ (generated by the {\tt sprand()} function) and $C=1$.  The column $B_1^T$ is a full random $n$-vector ($|B_1|=n$) stored in sparse format.  The sparse matrix $M$ is formed and then a symbolic factorization is computed.  The objective of this example is to demonstrate how the density of $B_2$ affects the upper-bound estimates for $L$ and $U$ formed in the numerical factorization of $M$, as well as the height of the column elimination tree.  All computations were performed using MATLAB R2014a.   Table \ref{tab:t1} presents statistics reported by {\tt symbfact()} and {\tt spparms('spumoni',2)} for $n=10^4$.  In the table
\begin{itemize}
\item  ``symb time'' represents the time required for {\tt symbfact()};
\item  $h$ is the height of the column elimination tree;
\item  ``$\overline{|L+U|}$'' represents the symbolic upper bound on the number of nonzero entries in $L+U$;
\item   ``$|L+U|$'' is the true number of nonzero entries once the actual LU factorization in executed, and
\item  ``solve time'' is the time required to execute \verb|x=M\b|, where \verb|b| is a sparse random $n+1$-vector.
\end{itemize}
All timing results are obtained using MATLAB's {\tt tic} and {\tt toc} functions and reported as seconds.
It is observed that the symbolic factorization time, the height of the column elimination tree, and the upper bound on nonzeros in $L+U$ grows in proportion to $|B_2|$ and not $|M|$.  As $|B_2|$ ranges from $0$ to $10^4$, the number of nonzeros in $M$ increases only by 50\%, while the amount of storage allocated for LU factorization increases by a factor of 2500, potentially affecting the performance of the direct solver.  In actuality, the time required for solving a linear system with coefficient matrix $M$ and the true number of nonzeros in $|L+U|$ scale more closely with $|M|$.  For much larger $n$ (as will be reported in Section \ref{sec:app}) however, this may not be the case.
\begin{table}\centering
\begin{tabular}{r|r|r|r|r|r|r}
$|M|$	&	$|B_2|$	&	symb time	&	$h$	&	$\overline{|L+U|}$	&	$|L+U|$	&	solve time	\\\hline
20001	&	0	&	0.0018	&	2	&	20001	&	20001	&	0.00019	\\
20002	&	1	&	0.0017	&	2	&	20002	&	20002	&	0.00036	\\
20011	&	10	&	0.0016	&	11	&	20056	&	20011	&	0.00034	\\
20100	&	99	&	0.0017	&	100	&	24951	&	20100	&	0.00038	\\
20945	&	944	&	0.0182	&	945	&	466041	&	20945	&	0.00060	\\
22211	&	2210	&	0.1059	&	2211	&	2463156	&	22211	&	0.00091	\\
23905	&	3904	&	0.3796	&	3905	&	7642561	&	23905	&	0.00139	\\
25956	&	5955	&	0.9637	&	5956	&	17753991	&	25956	&	0.00189	\\
30001	&	10000	&	2.7159	&	10001	&	50025001	&	30001	&	0.00285	
\end{tabular}
\caption{Solver statistics for $M$ with $n=10^4$ and increasing density of $B_2$.}
\label{tab:t1}
\end{table}

\section{Null Space Methods}\label{sec:ns}

Here the general null space method for \eqref{eq:e1} is outlined.  Further exposition and details can be found in \cite{ben051} for the case $B_1=B_2$ and $C=0$, and \cite{haslinger07} for $B_1\ne B_2$.  The first part of this section assumes $C=0$, and subsequently a variant of the null space method for $C\ne 0$ is described.

\subsection{The Null Space Method for \eqref{eq:e1}}
The following result (Theorem 3.1 of \cite{haslinger07}) establishes the equivalence of certain requirements on the blocks of $M$ in \eqref{eq:e1} with the invertibility of $M$ when $C= 0$.

\begin{theorem}\label{thm:Minv}
The matrix $M=\left[\begin{smallmatrix}
A&B_1^T\\B_2&0
\end{smallmatrix}\right]$ is invertible if and only if $B_1$ has full row rank, 
\begin{equation}
\dsp \ns(A)\cap\ns(B_2)=\{0\}, 
\label{eq:cond1}
\end{equation}
and 
\begin{equation}
\dsp \rng(A|\ns(B_2))\cap\rng(B_1^T)=\{0\}.
\label{eq:cond2}
\end{equation}
\end{theorem}

As remarked in \cite{haslinger07}, if  $B_1\in\real^{m\times n}$ has full row rank, then $\dim\rng(B_1^T)=m$, so \eqref{eq:cond2} implies $\dim  \rng(A|\ns(B_2))\le n-m$, and \eqref{eq:cond1} implies  $\dim \rng(A|\ns(B_2))=\dim\ns(B_2)\ge n-m$, and thus $\dim \ns(B_2)=n-m$.
This gives the following result.

\begin{corollary}\label{cor:b2}
If $M=\left[\begin{smallmatrix}
A&B_1^T\\B_2&0
\end{smallmatrix}\right]$ is invertible, then $B_2$ has full row rank.
\end{corollary}

The null space method requires
\begin{enumerate}
\item a particular solution $\bx^*$ of $B_2\bx=\bg$;
\item a matrix $Z_1\in \real^{n\times(n-m)}$ whose columns form a basis for $\ns(B_1)$; 
\item a matrix $Z_2\in \real^{n\times(n-m)}$ whose columns form a basis for $\ns(B_2)$.
\end{enumerate}

The algorithm proceeds by setting 
\begin{equation}
\bx:=Z_2\bv+\bx^*
\label{eq:xdef}
\end{equation}
so that 
\[B_2\bx=B_2(Z_2\bv+\bx^*)=B_2Z_2\bv+B_2\bx^*=B_2\bx^*=\bg,\]
as $B_2Z_2=0$.
Substituting $\bx$ into the first equation of \eqref{eq:e1} we have
\begin{equation}
AZ_2\bv+B_1^T\by=\bff-A\bx^*.
\label{eq:ns1}
\end{equation}
Premultiply \eqref{eq:ns1} with $Z_1^T$ to obtain
\[
Z_1^TAZ_2\bv+Z_1^TB_1^T\by=Z_1^T(\bff-A\bx^*),
\]
which reduces to, as $B_1Z_1=0$,
\begin{equation}
Z_1^TAZ_2\bv=Z_1^T(\bff-A\bx^*),
\label{eq:ns2}
\end{equation}
as $B_1Z_1=0$.  This problem has a unique solution, as shown below.

\begin{theorem}\label{thm:inv} Let $M=\left[\begin{smallmatrix}
A&B_1^T\\B_2&0
\end{smallmatrix}\right]$ be invertible and let the columns of $Z_1$ and $Z_2$ span the null spaces of $B_1$ and $B_2$, respectively. Then $Z_1^TAZ_2$ is invertible.
\end{theorem}

\begin{proof}  Let $\bv\in\real^{n-m}$ and assume that $Z_1^TAZ_2\bv=0$.  Then  $\bw=AZ_2\bv\in\rng(B_1^T)$ since $Z_1^T\bw=0$ implies $\bw$ is in the row space of $B_1$.  Now,  $B_2Z_2\bv=(B_2Z_2)\bv=0$ implies $Z_2\bv\in\ns(B_2)$, so  it is also the case that $\bw\in \rng(A|\ns(B_2))$.  Thus $\bw\in \rng(A|\ns(B_2))\cap \rng(B_1^T)$, and therefore Theorem \ref{thm:Minv}  implies \eqref{eq:cond2}, so $\bw=AZ_2\bv=0$.  This in turn implies $Z_2\bv\in\ns(A)$, and thus condition \eqref{eq:cond1} implies $Z_2\bv=0$.  Hence $Z_2^TZ_2\bv=0$, and since $\rank(Z_2^TZ_2)=\rank(Z_2)=\dim\ns(B_2)=n-m$, the invertiblity of $Z_2^TZ_2\in\real^{n-m}$ guarantees $\bv=0$, completing the proof.
\end{proof}

Once $\bv$ is found in \eqref{eq:ns2}, $\bx$ is obtained via \eqref{eq:xdef}, and $\by$ is solved by premultiplication of \eqref{eq:ns1} by $B_1$ to obtain the $m\times m$ system
\begin{equation}
B_1B_1^T\by = B_1\left(\bff-A(Z_2\bv+\xhat)\right),
\label{eq:ns3}
\end{equation}
which has a unique solution as $B_1B_1^T$ is invertible due to the fact that $B_1$ has full row rank.  The null space method is summarized in Algorithm \ref{algns}.

\begin{algorithm}[Null Space Method for $M$]
\label{algns}
{\rm \begin{enumerate}
\item Find $Z_1$ whose columns form a basis for the null space of $B_1$.
\item Find $Z_2$ whose columns form a basis for the null space of $B_2$.
\item Find $\bx^*$ such that $B_2\bx^*=\bg$. 
\item Solve $Z_1^TAZ_2\bv=Z_1^T(\bff-A\bx^*)$ for $\bv$.
\item Set $\bx = Z_2\bv+\bx^*$.
\item Solve $B_1B_1^T\by = B_1\left(\bff-A(Z_2\bv+\bx^*)\right)$ for $\by$.
\end{enumerate}}
\end{algorithm}

It should be noted that when $\bg=0$ (as is often the case in several applications), the particular solution $\xhat$ found in Step 3 of Algorithm \ref{algns} can simply be the trivial solution.
As remarked in the Introduction, the null space method is traditionally attractive when $n-m$ is small, as the systems in steps 4 and 6 of Algorithm \ref{algns} are both much smaller than  the $(n+m)\times (n+m)$ system \eqref{eq:e1}.  However, when $m$ is small, this method can be very efficient at transforming the $(n+m)\times (n+m)$ system $M\bv=\bb$ (where at least one of $B_1$ and $B_2$ are dense) to a reduced system of size $(n-m)\times (n-m)$ that retains the sparse structure  of $A$, perhaps making it more suitable for direct solvers.

\subsection{A One-Sided Null Space Method For $C\ne 0$}\label{sec:aug}
As noted in \cite{ben051}, the null space method cannot be applied directly to the situation when $M$ has a nonzero $(2,2)$ block, i.e., when $C\ne 0$ in \eqref{eq:e1}.
However, the null space method can be applied when the system \eqref{eq:e1} is augmented with an auxiliary variable $\bw$ to obtain
\begin{equation}
\hat{M}\hat{\bu}=\begin{bmatrix}
A&B_1^T&0\\B_2&C&0\\0&0&I_m
\end{bmatrix}\begin{bmatrix}
\bx\\\by\\\bw
\end{bmatrix}=\begin{bmatrix}
\bff\\\bg\\0
\end{bmatrix}.
\label{eq:mc2}
\end{equation}
It is clear that the invertibility of $M$ implies the invertibility of $\hat{M}$.  As opposed to eliminating both $B_1$ and $B_2$ from the system simultaneously, this can act as a ``one-sided'' null space method to eliminate one of $B_1$ or $B_2$  in one execution of Algorithm \ref{algns}.  Rewriting \eqref{eq:mc2} as
\begin{equation}
\begin{bmatrix}
A&B_1^T&0\\0&0&I_m\\B_2&C&0
\end{bmatrix}\begin{bmatrix}
\bx\\\by\\\bw
\end{bmatrix}=\begin{bmatrix}
\bff\\0\\\bg
\end{bmatrix},
\label{eq:mc3}
\end{equation}
and setting 
\[\xhat = \begin{bmatrix}
\bx\\\by
\end{bmatrix},\quad \hat{\bff}=\begin{bmatrix}
\bff\\0
\end{bmatrix},\quad  \hat{A}=\begin{bmatrix}
A&B_1^T\\0&0
\end{bmatrix},\quad \hat{B}_1^T = \begin{bmatrix}
0\\I_m
\end{bmatrix}, \quad\text{and}\quad \hat{B}_2=\begin{bmatrix}
B_2&C
\end{bmatrix},\]
\eqref{eq:mc3} is represented by the system
\begin{equation}
\begin{bmatrix}
\hat{A}&\hat{B}_1^T\\\hat{B}_2&0
\end{bmatrix}\begin{bmatrix}
\xhat\\\bw
\end{bmatrix}=\begin{bmatrix}
\hat{\bff}\\\bg
\end{bmatrix}. 
\label{eq:mc4}
\end{equation}
Since the coefficient matrix in \eqref{eq:mc4} is invertible, Theorem \ref{thm:inv} applies and therefore the null space method can be utilized.
Note that the null basis $\hat{Z}_1$ for $\hat{B}_1$ is simply 
\[\hat{Z}_1=\begin{bmatrix}
I_n\\0_{m\times n}
\end{bmatrix},\] and premultiplication of $\hat{A}$ by $\hat{Z}_1^T$ simply removes the $m$ zero rows from the bottom of $\hat{A}$.
The null space matrix of $\hat{B}_2$ can be written as $\hat{Z}_2 =\begin{bmatrix} Z_{2}^T &Z_C^T \end{bmatrix}^T$ where 
\begin{equation}
B_2Z_2+CZ_C=0.
\label{eq:schur1}
\end{equation}
Algorithm \ref{algns} then produces a linear system $\hat{Z}_1^T\hat{A}\hat{Z}_2\hat{\bv}=\hat{Z}_1^T(\hat{\bff}-\hat{A}\hat{\bx}^*)$ that reduces to
\begin{equation}
\left(AZ_2+B_1^TZ_{C}\right)\hat{\bv}=\bff-A\bx^*-B_1^T\by^*,
\label{eq:oneside}
\end{equation}
where $\hat{\bx}^* = \begin{bmatrix}\bx^* & \by^*\end{bmatrix}^T$ is the particular solution to $\hat{B}_2\hat{\bx}=\bg$.
with coefficient matrix $\hat{Z}_1^T\hat{A}\hat{Z}_2$ where $B_2$ and $C$  have been eliminated.  Note that when $C$ is invertible, \eqref{eq:schur1} implies $Z_C=-C^{-1}B_2Z_2$ and the coefficient matrix in \eqref{eq:oneside} is just the Schur complement of $C$ in $M$ times $\hat{Z}_2$.  When this is the case, $Z_2$ can be chosen such that $B_2Z_2=0$, thereby implying $Z_C=0$ and reducing \eqref{eq:oneside} to 
\begin{equation}
AZ_2\hat{\bv}=\bff-A\bx^*-B_1^T\by^*,
\label{eq:oneside2}
\end{equation}

While this one-sided approach may not be practical when $m\gg 1$, when $m$ is small this can be easily be employed to eliminate a small number of dense rows or columns.   Also note that there is no need to solve for $\bw$.  A summary of the procedure is given below:
\begin{algorithm}[One-Sided Null Space Method for $M$]
\label{algons}
{\rm 
\begin{enumerate}
\item Find $\hat{Z}_2=\begin{bmatrix} Z_{2} \\Z_C \end{bmatrix}$ whose columns form a basis for the null space of $\begin{bmatrix} B_2&C\end{bmatrix}$.
\item Find $\bx^*$ and $\by^*$ such that $B_2\bx^*+C\by^*=\bg$. 
\item Solve $\hat{A}\hat{Z}_2\hat{\bv}=\left(AZ_2+B_1^TZ_{C}\right)\hat{\bv}=\bff-A\bx^*-B_1^T\by^*$ for $\hat{\bv}$.
\item Set $\bx=Z_2\hat{\bv}+\bx^*$, $\by= Z_C\hat{\bv}+\by^*$.
\end{enumerate}}
\end{algorithm}

\vskip 0.1in

Alternatively, if the objective is to eliminate the column $B_1^T$ in \eqref{eq:mc2}, permute the last two columns of $\hat{M}$ so that $\hat{B}_1=\begin{bmatrix}
B_1&C^T
\end{bmatrix}$.

\section{The Null Space Method for Dense Rows and Columns}\label{sec:alg}
In this section the null space method to eliminate a dense row and/or column from $M$ in \eqref{eq:e1} will be described.  Here it is assumed that $M$ is invertible and $B_1=B_2=B$.  The algorithm for constructing $Z$ developed here can certainly be applied to the case $B_1\ne B_2$ for the null space method described in Algorithm \ref{algns}.

\subsection{Constructing $Z$ for $m=1$}\label{sec:makeZ}
To clearly illustrate the construction of the null space basis, in this section it will be assumed that $m=1$.
The task is to compute a basis for the $(n-1)$-dimensional subspace of $\real^n$ that is the null space of the row vector $B$.  For any $m$, a standard choice for $Z$ (from \cite{ben051}) is to permute the columns of $B$ by multiplication of a suitable permutation matrix $P$ to obtain $BP=\begin{bmatrix}
B_b&B_n
\end{bmatrix}$, where $B_b$ is $m\times m$ and nonsingular, and then construct $Z$ as 
\begin{equation}
Z=P\begin{bmatrix}
-B_b^{-1}B_n\\I_{n-m}
\end{bmatrix}.
\label{eq:badz}
\end{equation}
This ensures  that $BZ=0$.  Such a $Z$ is known as a {\em fundamental basis} \cite{pothenthesis}, and it is clearly desirable for $B_b$ to be easy to invert or factor.  As mentioned previously, this may be a difficult task for a larger $m$,  however, for the case $m=1$ considered here the construction of $Z$ can be computed in $\mathcal{O}(n)$ time.

As there must be at least one nonzero entry in $B=\begin{bmatrix}
b_1 & b_2 & \cdots & b_n
\end{bmatrix}$, without loss of generality assume $b_1\ne 0$.
Since the rank of $B$ is 1, $B_b$ in \eqref{eq:badz} can simply be set to $b_1$ and subsequently $-B_b^{-1}B_n$ is the $1\times (n-1)$ vector given by
\[B_b^{-1}B_n = \begin{bmatrix}
b_2/b_1, & b_3/b_1, &b_4/b_1, &\cdots, & b_n/b_1
\end{bmatrix}.\]  

However, this seemingly natural choice for $Z$ is not practical. 
While this construction is sparse (with exactly $(|B|-1)+(n-1)=n+|B|-2$ nonzero entries) it has a severe disadvantage when employing a null space method, as a dense $B$ implies the first row of $Z$ is also dense, which can have a drastic effect on the sparsity of products of the form $Z^T AZ$.  In fact, if $|B|=n$, the product $Z^TI_nZ=Z^TZ$ will be completely full, as is illustrated in Figure \ref{fig:badz}.  Since the outdegree of vertex $1'$ in $Z$ is $n$, the layered graph $L(Z^TZ)$ contains a path from every row index of $Z^T$ to every column index of $Z$.  Thus, the only situation in which $Z^TAZ$ will not be a full matrix is when there are no nonzero entries in the first column of $A$.
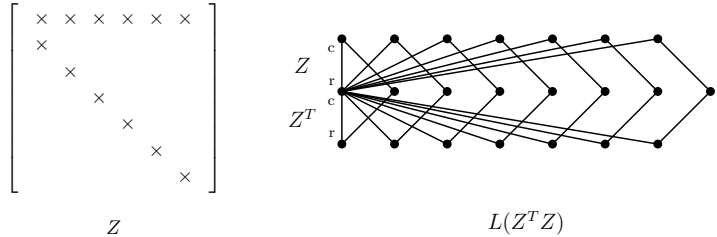
\begin{figure}[H]
\centering
\scalebox{0.75}{
\begin{tikzpicture}[scale=1.0]
\matrix (m)[ampersand replacement=\&,matrix of math nodes,left delimiter={[},right delimiter={]}]
{
\times \& \times \& \times \&\times \& \times \&\times  \\
\times \&  \& \&  \&   \&   \\
  \& \times \&   \&  \&  \&  \\
  \&   \& \times \& \&  \&   \\
  \&   \&   \&\times \&   \&  \\
  \&   \&   \&  \& \times \&  \\
  \&   \&   \&  \&  \&\times  \\
  };
   \node at(0,-2.3) {$Z$};
\end{tikzpicture}
}
\hskip 0.2in
\scalebox{0.7}{
\begin{tikzpicture}[style=thick]
\foreach \x in {0,1,2,3,4,5,6}
{\draw (0,1)--(\x,2);}
\foreach \x in {0,1,2,3,4,5,6}
{\draw (0,1)--(\x,0);}
\foreach \y in {1,2,3,4,5,6,7}
{\draw (\y,1) -- (\y-1,2);}
\foreach \y in {1,2,3,4,5,6,7}
{\draw (\y,1) -- (\y-1,0);}
\foreach \x in {1,2,3,4,5,6,7}{
\filldraw (\x-1,0) circle (2pt);
\filldraw (\x-1,2) circle (2pt);}
\foreach \x in {0,1,2,3,4,5,6,7}{
\filldraw (\x,1) circle (2pt);}
\node at(-0.75,0.5){\large $Z^T$};
\node at(-0.75,1.5){\large $Z$};
\node at(3.5,-1.5) {\large $L(Z^TZ)$};
\node[above left] at (0,0) {\footnotesize r};
\node[above left] at (0,1) {\footnotesize r};
\node[below left] at (0,1) {\footnotesize c};
\node[below left] at (0,2) {\footnotesize c};
\end{tikzpicture}}
\caption{Illustration of $Z$ and $Z^TZ$ in \eqref{eq:badz} when $|B|=n$.}
\label{fig:badz}
\end{figure}

Thus, an alternate method for constructing the null space basis $Z$ that allows $Z^TAZ$ to retain as much of the original sparse structure of $A$ as possible is desired.  To achieve this, $Z$ should be constructed so that the outdegree of all rows of $Z$ is as small as possible.  Following the general principle that each column $\bz_j$ of $Z$ must be constructed so that $B\bz_j=0$, an alternate approach can be derived from using only a single nonzero entry in $B$ to eliminate a subsequent nonzero entry.  Again, without loss of generality, assume that $b_1\ne 0$, and also let $b_m$ represent the last nonzero entry in $B$ ($m\le n$).  The first step is to find the next nonzero entry in $B$, as it will be used with $b_1$ to produce a null vector.  Starting with $j=1+1$, if $b_j=0$, then column $\bz_j$ of $Z$ will simply be set to $\be_j\in\real^n$ (the elementary basis vector for index $j$), as this will imply $B\bz_j=0$.  If a nonzero $b_j$ is found, then the first column $\bz_1$ of $Z$ is all zeros with the exception of a $1$ in entry 1 and $-b_1/b_j$ in entry $j$.  This ensures that $B\bz_1=0$.  
 The algorithm resumes by setting $i$ equal to $j$ and then seeking the next nonzero entry in $B$.  The above procedure is repeated until the last nonzero entry $b_m$ is encountered.  At this point, $m-1$ linearly independent columns $\bz_k$,  $k=1,\ldots,m-1$, have been constructed, and $B\bz_k=0$ for each $k$.  It remains to add $n-m$ more linearly independent columns to $Z$.  Since the $n-m$ entries $b_{m+1}=b_{m+2}=\cdots=b_n=0$, simply add the elementary basis vector $\be_k$, $m+1\le k\le n$ to complete the construction of $Z$.
The procedure, which requires $\mathcal{O}(n)$ operations, is summarized in Algorithm \ref{alg1}.  In the algorithm, $\varepsilon$ represents the smallest number that is to be treated as a numerical nonzero.  Depending on the programming environment and sparse data structures employed, the algorithm can be executed very quickly, for example using the {\tt find()} function and vectorized computations in MATLAB.

\begin{algorithm}[Construct $Z$]  
{\rm
\begin{algorithmic}[1]
\State $Z\gets 0_{n\times (n-1)}$
\State $i\gets 1$
\While{$i < m$}
\State $Z_{i,i}\gets 1$
\State $j\gets i+1$
\While{$|b_j|< \varepsilon$}
\State $Z_{j,j}\gets 1$
\State $j\gets j+1$
\EndWhile
\State $Z_{j,i} \gets -b_i/b_j$
\State $i\gets j$
\EndWhile
\For{$i=m$ to $n$}
\State $Z_{i+1,i} \gets 1$
\EndFor
\end{algorithmic}}
\label{alg1}
\end{algorithm}

The result of the preceding discussion is summarized below.
\begin{lemma}
For any $B\in \real^{1\times n}$ with at least one nonzero entry, the columns of $Z$ produced by Algorithm \ref{alg1} form a basis for the $(n-1)$-dimensional null space of $B$.  The number of nonzero entries in $Z$ is $|B|+n-2$.
\end{lemma}

While the $Z$ constructed by Algorithm \ref{alg1} has exactly the same number of nonzero entries as the $Z$ given by \eqref{eq:badz}, the important characteristic of this $Z$ is that each row and each column have at most 2 nonzero entries.
An example of a $B$ and its corresponding $Z$ is given in \eqref{eq:bex}, and the general form of $Z$ and $L(Z^TZ)$ when $|B|=n$ are given in Figure \ref{fig:goodz2}.
\begin{equation}
\begin{footnotesize}
B^T=\begin{bmatrix}
0 \\1\\ -3\\ 0\\ -1\\ 2\\ 0\\ 0
\end{bmatrix}\qquad\Longrightarrow\qquad Z=\begin{bmatrix}
1 & & & & & &   \\
& 1 & & & & &   \\
&\frac{1}{3}  & 1 &  & &  \\
& &  &1 & & &  \\
& & -3 & &1 & &   \\
& &  & & \frac{1}{2}& &   \\
& &  & & & 1&   \\
& &  & & & &   1
\end{bmatrix}\end{footnotesize}
\label{eq:bex}.
\end{equation}
\begin{figure}[H]
\centering
\scalebox{0.75}{
\begin{tikzpicture}[scale=1.0]
\matrix (m)[ampersand replacement=\&,matrix of math nodes,left delimiter={[},right delimiter={]}]
{
\times \& \& \& \&  \&  \\
\times \& \times \& \&  \&   \&   \\
  \& \times \&  \times \&  \&  \&  \\
  \&   \& \times \& \times \&  \&   \\
  \&   \&   \&\times \&  \times \&  \\
  \&   \&   \&  \& \times \&  \times\\
  \&   \&   \&  \&  \&\times  \\
  };
   \node at(0,-2.3) {$Z$};
\end{tikzpicture}
}
\hskip 0.2in
\scalebox{0.7}{
\begin{tikzpicture}[style=thick]
\foreach \x in {0,1,2,3,4,5,6}
{\draw (\x,1)--(\x,2);\draw (\x,0)--(\x,1);}
\foreach \x in {0,1,2,3,4,5,6}
{\draw (\x+1,1)--(\x,2);\draw (\x+1,1)--(\x,0);}
\foreach \x in {1,2,3,4,5,6,7}{
\filldraw (\x-1,0) circle (2pt);
\filldraw (\x-1,2) circle (2pt);}
\foreach \x in {0,1,2,3,4,5,6,7}{
\filldraw (\x,1) circle (2pt);}
\node at(-0.75,0.5){\large $Z^T$};
\node at(-0.75,1.5){\large $Z$};
\node at(3.5,-1.5) {\large $L(Z^TZ)$};
\node[above left] at (0,0) {\footnotesize r};
\node[above left] at (0,1) {\footnotesize r};
\node[below left] at (0,1) {\footnotesize c};
\node[below left] at (0,2) {\footnotesize c};
\end{tikzpicture}}
\caption{Illustration of $Z$ and $Z^TZ$ given by Algorithm \ref{alg1} when $|B|=n$.}
\label{fig:goodz2}
\end{figure}
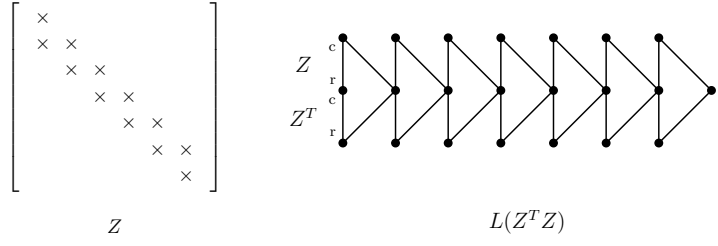

\begin{remark}
{\rm The maximum degree of any row or column vertex in $H(Z)$, where $Z$ is generated by Algorithm \ref{alg1}, is 2.}
\end{remark}

Figure \ref{fig:goodz} gives an example of how the nonzero pattern of $Z^TAZ$ is determined when $B$ is full.  In this worst-case situation, while $Z^TAZ$ may have as many as 4 times as many nonzero entries as $A$, the product retains the overall sparse structure of $A$ and the dense row and column are eliminated from the system.
\begin{figure}[h]
\centering
\scalebox{0.75}{
\begin{tikzpicture}[style=thick]
\draw (0,1)--(2,2);
\draw (0,1)--(5,2);
\draw (1,1)--(3,2);
\draw (1,1)--(4,2);
\draw (2,1)--(0,2);
\draw (2,1)--(6,2);
\draw (3,1)--(1,2);
\draw (4,1)--(2,2);
\draw (4,1)--(7,2);
\draw (5,1)--(4,2);
\draw (6,1)--(9,2);
\draw (7,1)--(6,2);
\draw (8,1)--(5,2);
\draw (9,1)--(10,2);
\draw (9,1)--(7,2);
\draw (10,1)--(9,2);
\foreach \x in {0,1,2,3,4,5,6,7,8,9}
{\draw (\x,3)--(\x,2);\draw (\x,0)--(\x,1);}
\foreach \x in {0,1,2,3,4,5,6,7,8,9}
{\draw (\x,3)--(\x+1,2);\draw (\x,0)--(\x+1,1);}
\foreach \x in {0,2,4,5,6,8,9}
{\draw (\x,2)--(\x,1);}
\foreach \x in {0,1,2,3,4,5,6,7,8,9}{
\filldraw (\x,0) circle (2pt);
\filldraw (\x,3) circle (2pt) ;}
\foreach \x in {0,1,2,3,4,5,6,7,8,9,10}{
\filldraw (\x,1) circle (2pt);
\filldraw (\x,2) circle (2pt);}
\node[above left] at (0,0) {\footnotesize r};
\node[above left] at (0,1) {\footnotesize r};
\node[above left] at (0,2) {\footnotesize r};
\node[below left] at (0,1) {\footnotesize c};
\node[below left] at (0,2) {\footnotesize c};
\node[below left] at (0,3) {\footnotesize c};
\node at(-0.75,0.5){\large $Z^T$};
\node at(-0.75,1.5){\large $A$};
\node at(-0.75,2.5){\large $Z$};
 \node at(5,-0.75) {\large $L(Z^TAZ)$};
 \filldraw[white,fill=white] (6.2,-0.25) rectangle (6.8,3.25);
 \draw (6.2,3.5)--(6.2,-0.5);
\draw (6.8,3.5)--(6.8,-0.5);
\node at (6.5,1.5){$\cdots$};
\end{tikzpicture}}
\caption{The layered graph $L(Z^TAZ)$ for $Z$ constructed by Algorithm \ref{alg1} when $|B|=n$ and a random $A$.}
\label{fig:goodz}
\end{figure}
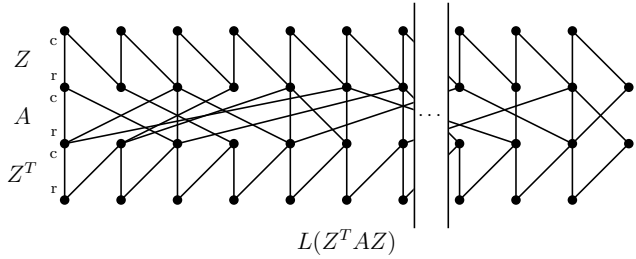

\begin{theorem}\label{thm:bound}
Assume $m=1$, $C=0$, and $B_1=B_2=B$ in \eqref{eq:e1}.  Let $Z\in\real^{n\times(n-1)}$ be given as in Algorithm \ref{alg1}.  Then  $Z^TAZ$ in Step 4 of Algorithm \ref{algns} satisfies $|Z^TAZ| \le 4|A|$.
\end{theorem}
\begin{proof}
The result is shown by bounding the number of nonzero entries in row $i$ of $Z^TAZ$ and then summing over the rows.
Let $i\in 1,\ldots,n-1$ be a row vertex  of $Z^T$.  The construction of $Z$ guarantees that the outdegree of vertex $i$ at most 2.  Assume the outdegree is 2 and let $j_{i,1}, j_{i,2}$ be the column vertices of $Z^T$ reached by vertex $i$ in $H(Z^T)$.  These column vertices of $Z^T$ correspond to the row vertices $j_{i,1}$, $j_{i,2}$ of $A$, which have outdegree $|A_{j_{i,1}*}|$, $|A_{j_{i,2}*}|$ respectively.  Each column vertex of row $j_{i,1}$, $j_{i,2}$ of $A$ corresponds to a row vertex of $Z$, all of which have an outdegree of at most 2.  Thus, the number of possible paths originating from row vertex $i$ in the layered graph of $Z^TAZ$ is bounded by
\[|(Z^TAZ)_{i*}|\le 2|A_{j_{i,1}*}|+2|A_{j_{i,2}*}|.\]
Using part 1 of Lemma \ref{lem:cohen} and summing over the rows of $Z^T$ gives the bound
\begin{equation}
|Z^TAZ| \le \sum_{i=1}^{n-1} \left(2|A_{j_{i,1}*}|+2|A_{j_{i,2}*}|\right).
\label{eq:ztazb1}
\end{equation}
However, the indegree of any column vertex of $Z^T$ is at most 2 as well, which guarantees that the indices $j_{i,1}$ and $j_{i,2}$ appear at most twice in the right hand side of \eqref{eq:ztazb1}.  Thus we have
\begin{equation}
|Z^TAZ| \le \sum_{i=1}^{n-1} \left(2|A_{j_{i,1}*}|+2|A_{j_{i,2}*}|\right)\le \sum_{j=1}^n 2\cdot 2|A_{j*}| = 4|A|,
\label{eq:ztazb2}
\end{equation}
which completes the proof.
\end{proof}

\begin{remark}{\rm
The maximum inflation, i.e. $|Z^TAZ|/|M|$ when $|B|=n$ is  $4-8n/|M|$.  For example, if $A=I_n$, $|M|=3n$ which results in a maximum inflation of $4-8/3=4/3$.}
\end{remark}

Recalling the one-sided null space method in Section \ref{sec:aug}, a similar argument to Theorem \ref{thm:bound} leads to the following result.  Note that the last row of $\hat{Z}_2$ produced by Algorithm \ref{alg1} contains only one nonzero entry.
\begin{theorem}  
Let $\hat{B}_2=\begin{bmatrix} B_2&C\end{bmatrix}$ in \eqref{eq:e1}.  Let $\hat{Z}_2\in\real^{n+1\times n}$ be given as in Algorithm \ref{alg1}.  Then  $\hat{A}\hat{Z}_2$ in Step 3 of Algorithm \ref{algons} satisfies $|\hat{A}\hat{Z}_2| \le 2|A|+|B_1|$.
\end{theorem}

\subsection{Conditioning vs. Sparsity}

As remarked in \cite{ben051}, there is a trade-off between good conditioning properties of $Z$ and sparsity.     While
Algorithm \ref{alg1} can easily be modified so that each column is normalized as it constructed, any orthogonalization process such as Gram-Schmidt on the columns of $Z$ will induce significant fill above (or below) the diagonal.  While it is possible to construct an orthogonal basis for $\ns(B)$ from the last $n-1$ columns of $Q$ in the QR factorization of $B$, in general this is impractical when $B$ is merely a single row and $n$ is large.

However, as the intent of this null space method is to prepare a matrix for a direct solver, conditioning may not necessarily be as significant of a concern as it is when an iterative solver is employed.  It is still important to understand the general conditioning properties of $Z$.  For the moment, assume $B$ is full, and let $\alpha_i = -b_i/b_{i+1}$ (i.e., the subdiagonal entries of $Z$).  A simple upper bound estimate for the largest singular value $\sigma_1(Z)$ of $Z$ can be found using Corollary 2.4.4 of \cite{golubvanloan}.  Writing $Z=I_{n\times (n-1)}+Z_\alpha$, we have
\[\sigma_1(Z)=\sigma_1(I_{n\times (n-1)}+Z_\alpha) \le \|I_{n\times (n-1)}\|_2+\sigma_1(Z_\alpha)=1+\max_{1\le i\le n-1}\{|\alpha_i|\},\]
where $\|\cdot\|_2$ is the matrix 2-norm.  This indicates that the relative sizes of successive nonzero entries  in $B$ play a major role in the conditioning of $Z$, and suggests that permuting the rows/columns of $M$ to achieve a smaller maximum $|\alpha|$ would be advantageous.  However, such an approach could destroy any desirable structure properties that $A$ enjoys, such as a small bandwidth.

Lower bounds for the smallest singular value $\sigma_{n-1}(Z)$ are not as easy.  While several bounds exist in the current literature, the seemingly sharpest bounds to date \cite{zou2010,zou2012} require quantities raised to powers on the order of $n$, which is not feasible for large sparse problems.  Of course, conditioning is a well-studied problem, therefore an efficient and reliable condition estimator such as MATLAB's {\tt condest()} applied to $Z^TZ$ should be utilized to gauge the potential impact of the null space method on  $Z^TAZ$.
When the condition number of $Z^TAZ$ is very large, gains in solver performance may still be realized by employing the one-sided null space method in Algorithm \ref{algons}.

\subsection{Extension to $m>1$}

The following slight generalization of Theorem 6.4.1 in \cite{golubvanloan} allows for an easy extension of Algorithm \ref{alg1} to the case $m>1$.

\begin{lemma} Let $B_1\in \real^{m_1\times n}$ and $B_2\in\real^{m_2\times n}$ have full row rank.  Assume the columns of $Z_1\in\real^{n\times (n-m_1)}$ span $\ns(B_1)$, and assume the 
columns of $Z_2\in\real^{(n-m_1)\times (n-m_1-m_2)}$ span $\ns(B_2Z_1)$.  Then the columns of $Z_1Z_2$ form a basis for the null space of $\begin{bmatrix}
B_1\\B_2
\end{bmatrix}$.
\end{lemma}

The nested construction of $Z^{n\times (n-m)}$ is given in the algorithm below where $B_i$ is the $i$th row of the $m\times n$ block $B$.

\begin{algorithm}[Construct $Z$ for $m>1$]  {\rm
\begin{algorithmic}[1]
\State Construct $Z_1$ for $B_1$ using Algorithm \ref{alg1}.
\For{$i =2$ to $m$}
\State Construct $Z\in \real^{(n-i+1)\times (n-i)}$ for $B_iZ_{i-1}\in \real^{1\times (n-i+1)}$ using Algorithm \ref{alg1}.
\State $Z_i\gets Z_{i-1}Z$
\EndFor
\end{algorithmic} }
\label{alg2}
\end{algorithm}

Algorithm \ref{alg2} constructs $Z$ in $m$ applications of Algorithm \ref{alg1}, $m-1$ matrix-vector products, and $m-1$ matrix multiplications, each of which are the product of two lower triangular matrices.

\section{Applications of the Method}\label{sec:app}
Here several applications of the null space method for systems with dense rows and columns are described.  In these applications, the null space basis matrix $Z$ is constructed as given in Section \ref{sec:makeZ}.  Numerical results that compare the performance of the null space method versus the standard direct solver are given.

All computations were performed on a Mac Pro with a 3.33 GHz 6-core processor and 32 GB of memory.  All matrices arising in finite element methods were assembled using the FreeFEM++ environment \cite{hecht}.  All linear systems were solved using MATLAB R2014a.  Compiled statistics include the number of nonzeros in $M$, $B_2$ (unless obvious), and $Z_1^TAZ_2$, as well as
\begin{itemize}
\item {\em infl}, the inflation, which is $|Z_1^TAZ_2|/|M|$;
\item {\em diff}, which is $\|\bx-\bx_*\|_\infty/\|\bx\|_\infty$, where $\bx$ is the solution computed by the standard solution \verb|x=M\b| and $\bx_*$ is the solution computed by the null space method;
\item {\em Ztime}, which is the time required to construct $Z_1$ and $Z_2$;
\item {\em NStime}, which is the time required for the null space method set up and solution (includes Ztime);
\item {\em Stime}, which is the time required to execute the direct solve \verb|x=M\b|;
\item and {\em speedup}, which is Stime/NStime.
\end{itemize}
The MATLAB {\tt tic} and {\tt toc} functions were used to capture all timing data.
\subsection{ Arrowhead Matrices and the One-Sided Null Space Method}
In Section \ref{sec:sym2} arrowhead matrices were employed to demonstrate the impact of a dense row on direct solver heuristics.  Let $m=1$,  $A=I_n$, $B_1$ and $B_2$ be full vectors (in sparse format) with random entries between 0 and 1, and $C=1$.  The right-hand side vector $\bb$ is also full (in sparse format) with random values between 0 and 1.  When the one-sided null space method described in Algorithm \ref{algons} using the null space basis given in Algorithm \ref{alg1} is applied, the $(n+1)\times(n+1)$ arrowhead matrix is transformed into an $n\times n$ upper triangular matrix, which can be solved very efficiently.  Table \ref{tab:t2} in Appendix \ref{sec:Aarrow} gives statistics obtained for increasing $n$, while the plot on the left in Figure \ref{fig:arrow} shows the run times of both methods.
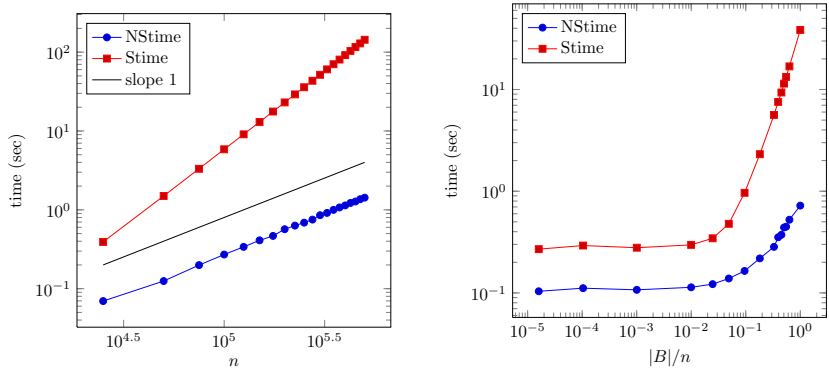
\begin{figure}
\centering
\begin{tikzpicture}[scale=0.65]
\begin{axis}[
width=8cm, height=8cm,
    xmode=log,ymode=log, xlabel={$n$}, ylabel={time (sec)},
    legend entries={NStime, Stime,slope 1},
legend style={nodes=right},
legend pos= north west]
]
\addplot table {
25001	0.06967074
50001	0.125021985
75001	0.19905139
100001	0.271162636
125001	0.339092499
150001	0.410425819
175001	0.466895723
200001	0.568036897
225001	0.630988437
250001	0.68783542
275001	0.751764914
300001	0.855101807
325001	0.913874553
350001	0.998929062
375001	1.072482412
400001	1.137568578
425001	1.222753694
450001	1.27554701
475001	1.36397097
500001	1.427434822
};
\addplot table {
25001	0.392161799
50001	1.498658924
75001	3.317684656
100001	5.857778124
125001	9.066225146
150001	12.9993712
175001	17.64910252
200001	23.02132101
225001	29.10682217
250001	35.87434119
275001	43.37869859
300001	51.61237807
325001	60.4502475
350001	70.2628186
375001	80.48689708
400001	91.61732804
425001	103.6698422
450001	116.1098147
475001	129.0522355
500001	143.5631473
};
\addplot[color=black,solid] table { 
25000	0.2
50000	0.4
100000	0.8
200000	1.6
500000 4.0
};
\end{axis}
\end{tikzpicture}
\hskip 0.2in
\begin{tikzpicture}[scale=0.65]
\begin{axis}[
width=8cm, height=8cm,
    xmode=log,ymode=log, xlabel={$|B|/n$}, ylabel={time (sec)},
    legend entries={NStime, Stime},
legend style={nodes=right},
legend pos= north west]
]
\addplot table {
0.000016	0.103778874
0.000104	0.111338215
0.001004	0.107270915
0.00996	0.113511066
0.0247	0.122020189
0.048876	0.138618902
0.095164	0.164557251
0.181256	0.218534058
0.329436	0.283628878
0.393308	0.352506763
0.450736	0.372760607
0.503208	0.439739155
0.550144	0.447778345
0.632088	0.524410645
1	0.721726573
};
\addplot table {
0.000016	0.269328282
0.000104	0.292294216
0.001004	0.278336501
0.00996	0.296802856
0.0247	0.344320769
0.048876	0.47825459
0.095164	0.96473495
0.181256	2.311932343
0.329436	5.62122231
0.393308	7.558008169
0.450736	9.338915158
0.503208	11.42157258
0.550144	13.29916063
0.632088	16.87130355
1	38.47906043
};
\end{axis}
\end{tikzpicture}
\caption{Log-scale plot of running times for the one-sided null space method and the standard direct solve for arrowhead matrices for increasing $n$ with $|B_1|=n$ (left) and increasing $|B|/n$ with $n=250000$ (right) .}
\label{fig:arrow}
\end{figure}
Similar to the example shown in Section \ref{sec:sym2}, the effect on the increasing density of a row was also analyzed for both methods.  Again, $A=I_n$ and $C=1$, however here $B_1=B_2=B$ is a sparse random vector with entries between 0 and 1 with a varying number of nonzero entries.  Results for $n=250000$ are given in Table \ref{tab:arrow2}, and a plot of the running times for both methods vs. increasing $|B|/n$ is given on the right  in \ref{fig:arrow}.  The standard direct solve time begins to increase sharply when around $10\%$ of the entries in $B$ are nonzero and experiences an overall increase in running time of more than two orders of magnitude over the range of $|B|$.  In contrast, the one-sided null space method run times increase less than one order of magnitude over the entire range of $|B|$.

\subsection{Elimination of Mean Zero Constraints in Finite Element Approximations}
Variational problems arising from PDEs often require unknowns to be in specific subspaces of function spaces in order to be well-posed.  A very common example of this is the condition that the mean value of the unknown over the spatial domain is zero.  When the finite element method is employed to approximate solutions to these problems, conditions such as the mean zero constraint must also be imposed on the finite element subspaces used for approximating the unknowns.  

Several software packages and environments are available to implement the finite element method.  In many of the more recent packages, the end user merely needs to specify the computational mesh, specify which finite element spaces will be used, and provide the variational form of the problem.  While software with these capabilities allows for great flexibility in a very short development time frame, it places certain limitations on the control available to the user.  For example, it may be difficult to specify that an unknown is to be restricted to a subspace with a mean zero constraint, and it may be tedious or inefficient to modify the linear systems assembled by the software.  Thus, it is often the case that the best way to implement a condition such as the mean zero constraint is via the use of a scalar Lagrange multiplier.  In this case, it is straightforward to augment the original linear system produced by the method with an additional row and column that represent the additional degree of freedom and the constraint it represents.  In many cases these augmented rows and columns are dense (if not full).

Typically $\Omega$ will represent a spatial domain with boundary $\Gamma$.  Other notation used in this section includes the usual $L^2(\Omega)$ label for the Lebesgue space of square-integrable functions defined on a spatial domain $\Omega\subset\real^d$ ($d=2,3$) with Lipschitz
continuous boundary $\Gamma$, and $H^1(\Omega)$ for those in $L^2(\Omega)$ with gradients also in $(L^2(\Omega))^d$.  The space $H(\dv,\Omega)$ represents the functions in $(L^2(\Omega))^2$ with divergence in $L^2(\Omega)$.  A subscript of zero on a space indicates that it is the subspace of functions with zero mean. The pairing $(\cdot,\cdot)$ represents the $L^2$ inner product.

Three examples are presented, each demonstrating a different aspect of the behavior of the method.  First, the performance of the null space method is compared to the standard solve for a problem with $|B|=n$ and two different nonzero patterns.  Subsequently, a comparison of the null space method's performance on the Stokes and Navier-Stokes problems on the same domain and mesh is given (with $|B|$ around 11\% of $n$).  Finally, a comparison of the null space method's performance on a dual-mixed problem with a regular nonzero pattern is compared to similar systems with an irregular nonzero pattern (with $|B|$ around 40\% of $n$).  A summary of the comparisons between the null space method and the standard direct solve for largest problem of each of the different examples in given in Table \ref{tab:sum}.
\begin{table}[H]\centering
\begin{small}
\begin{tabular}{r|r|r|r|r|r|r}
Problem	&	Domain	&	$n+1$	&	Inflation	&	NStime	&	Stime	&	Speedup	\\\hline
Poisson	&	Square	&	303602	&	1.57	&	4.717	&	108.684	&	23.04	\\
Poisson	&	Ring-like	&	473786	&	1.96	&	36.037	&	60.448	&	1.68	\\
Mixed Stokes	&	Square	&	394189	&	1.33	&	81.795	&	1349.064	&	16.49	\\
Mixed Navier-Stokes	&	Square	&	394189	&	1.21	&	136.018	&	1357.513	&	9.98	\\
Dual-Mixed Stokes	&	Square	&	720961	&	1.03	&	105.158	&	2926.810	&	27.83	\\
Dual-Mixed Stokes	&	Irregular	&	656314	&	1.03	&	38.662	&	1604.043	&	41.49	
\end{tabular}
\end{small}
\caption{Comparison of standard direct solve and null space method for largest problem of each type.}
\label{tab:sum}
\end{table}

\subsubsection{Poisson Problem with Pure Neumann Conditions}
When the Poisson problem $-\Delta u=f$ in $\Omega$ is augmented with a pure Neumann boundary condition $\nabla u\cdot n=g$ on $\Gamma$, it is only well-posed provided the solution has mean zero over the spatial domain \cite{bre941}.   Employing a scalar Lagrange multiplier to enforce this, the variational problem results in the system 
\[\begin{bmatrix}
A&B^T\\B&0
\end{bmatrix}\begin{bmatrix}
u\\\lambda
\end{bmatrix}=\begin{bmatrix}
f\\0
\end{bmatrix},\]
where $A$ is large and sparse with a (usually small) bandwidth, and $B\in \real^{1\times n}$ is dense, representing $\int_\Omega u \,d\Omega$.  In this example, two different spatial domains $\Omega$ are used: the first is the unit square in $\real^2$ and the finite element mesh is regular, while the second is a ring-like domain with a Delaunay mesh.  Representatives of both meshes are given in Figure \ref{fig:lapmesh}, and example sparsity plots for both problems are given in Figure \ref{fig:lapspy}.   Continuous piecewise linear finite elements were employed to approximate $u$, and $|B|=n$ (completely full).  
\begin{figure}
\centering
\includegraphics[scale=0.2]{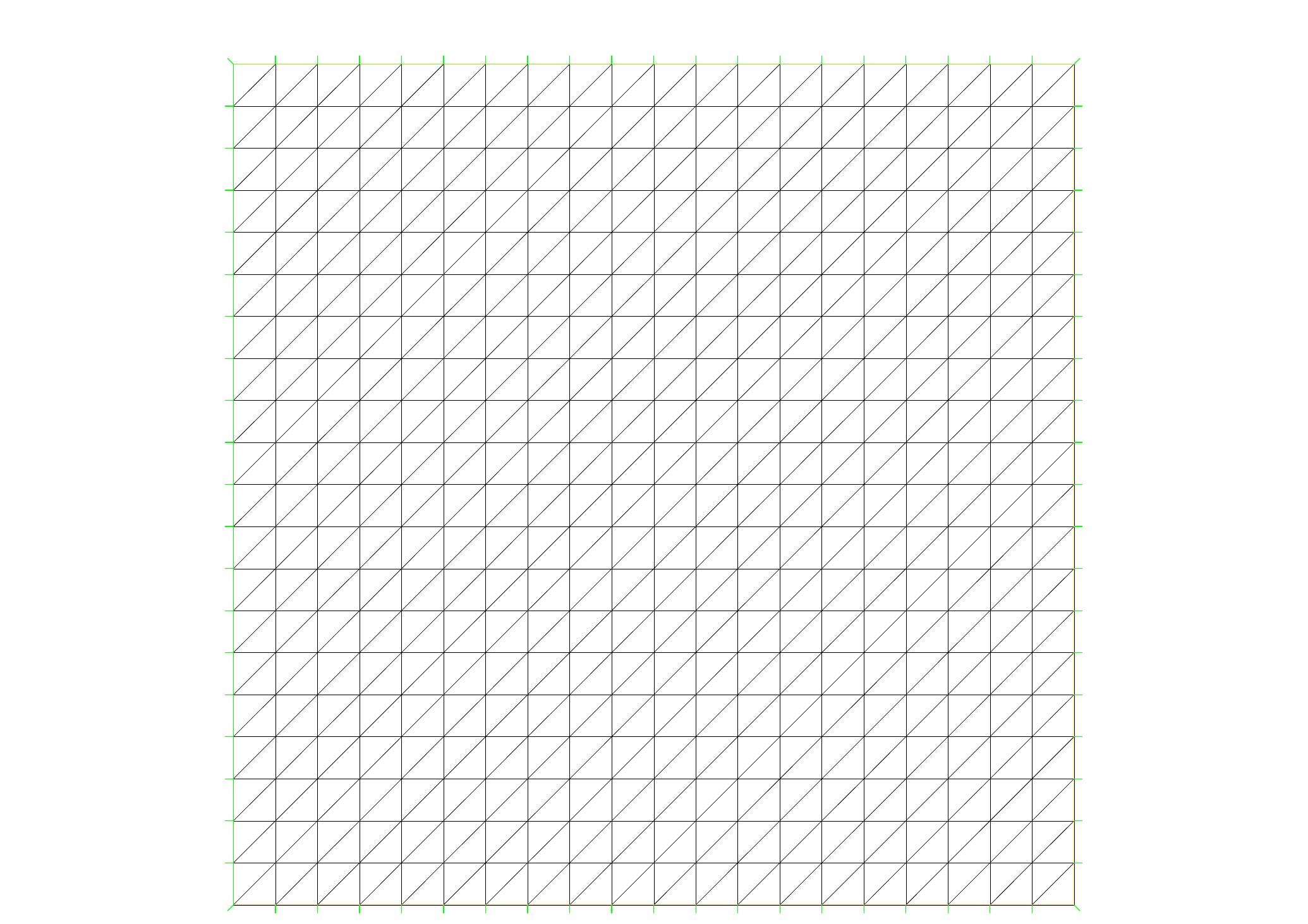}
\includegraphics[scale=0.2]{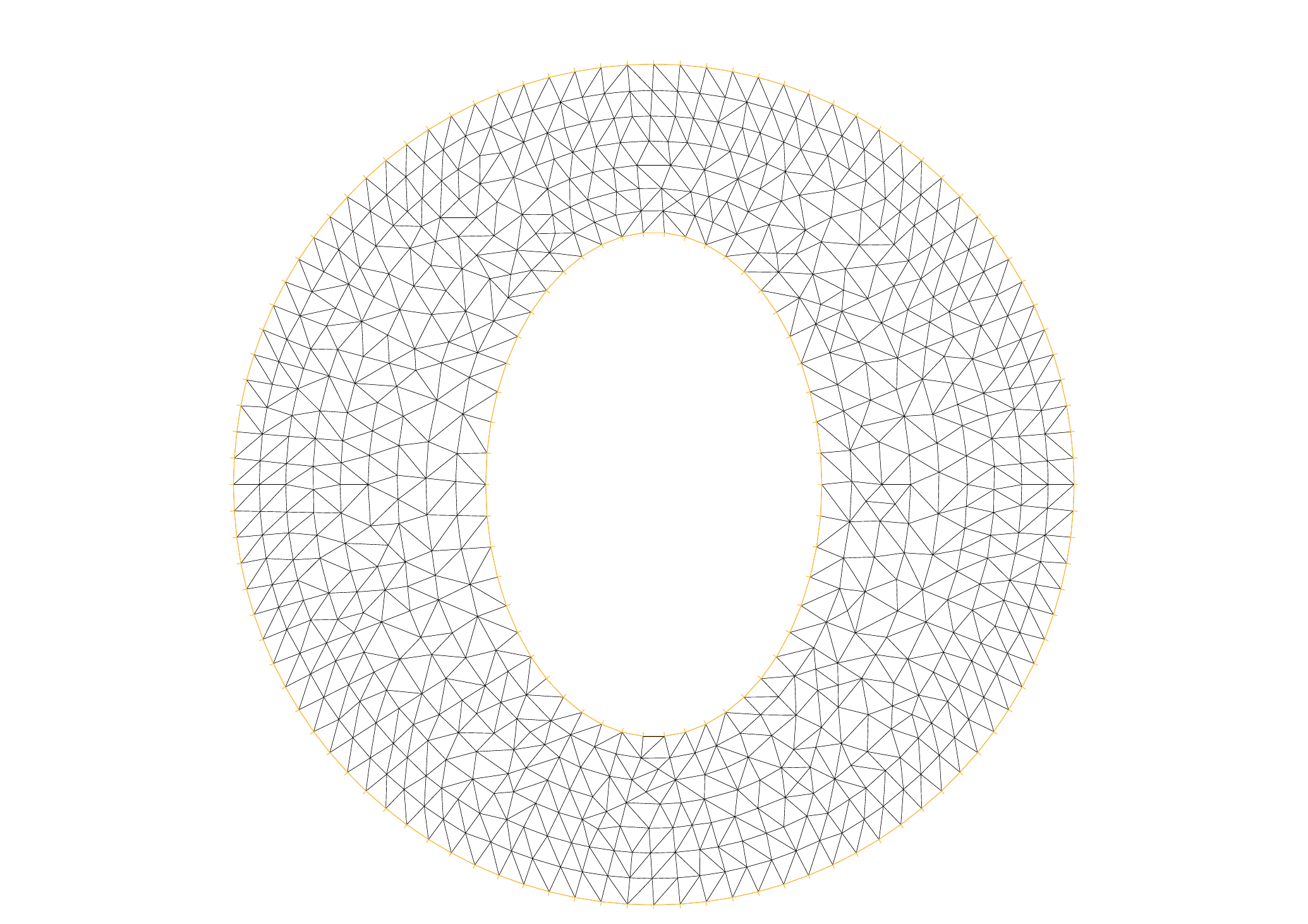}
\caption{Example meshes used for Poisson problem with square domain (left) and  ring-like domain (right).}
\label{fig:lapmesh}
\end{figure}
\begin{figure}
\centering
\includegraphics[scale=0.25]{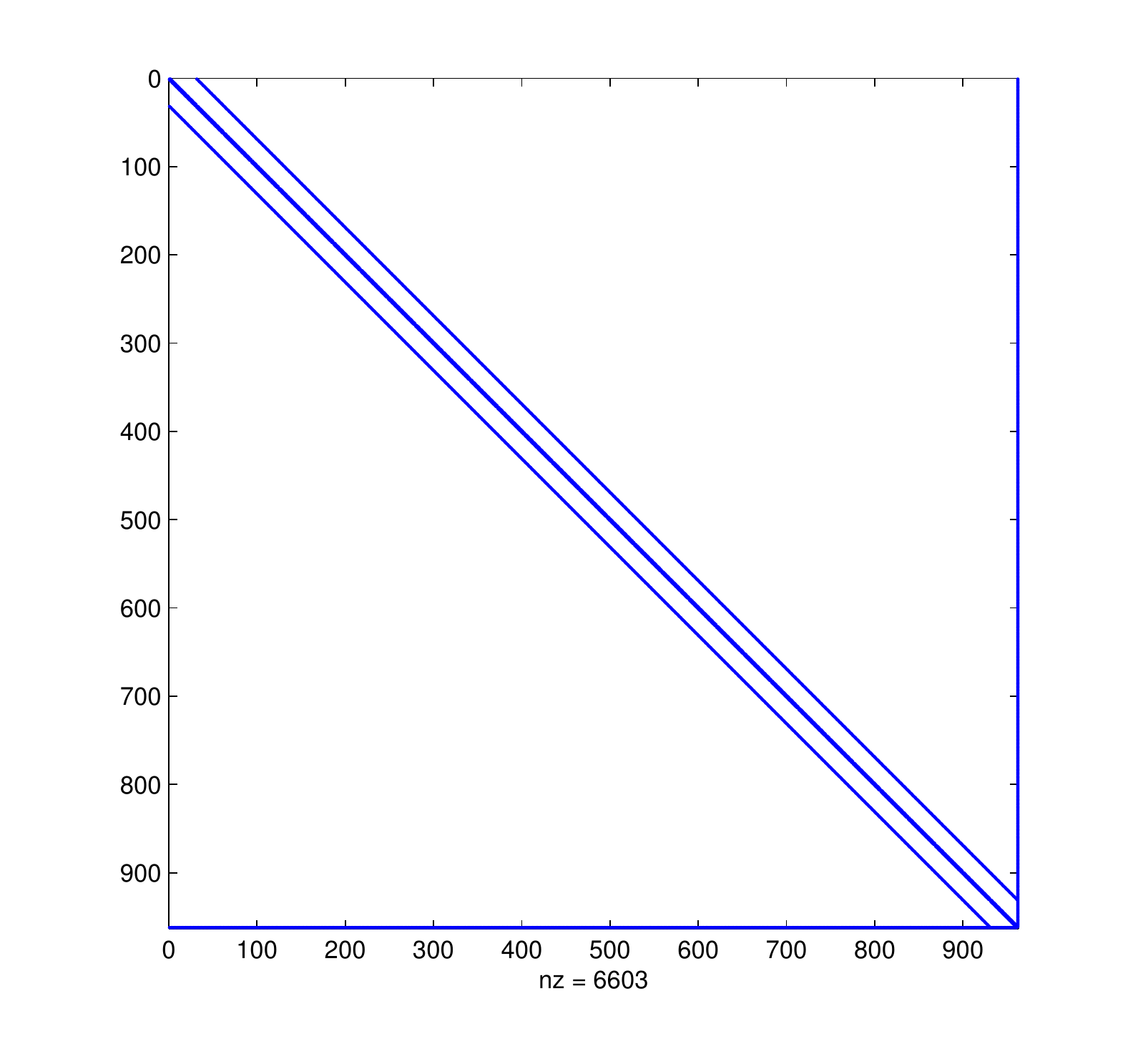}
\hskip 0.25in
\includegraphics[scale=0.25]{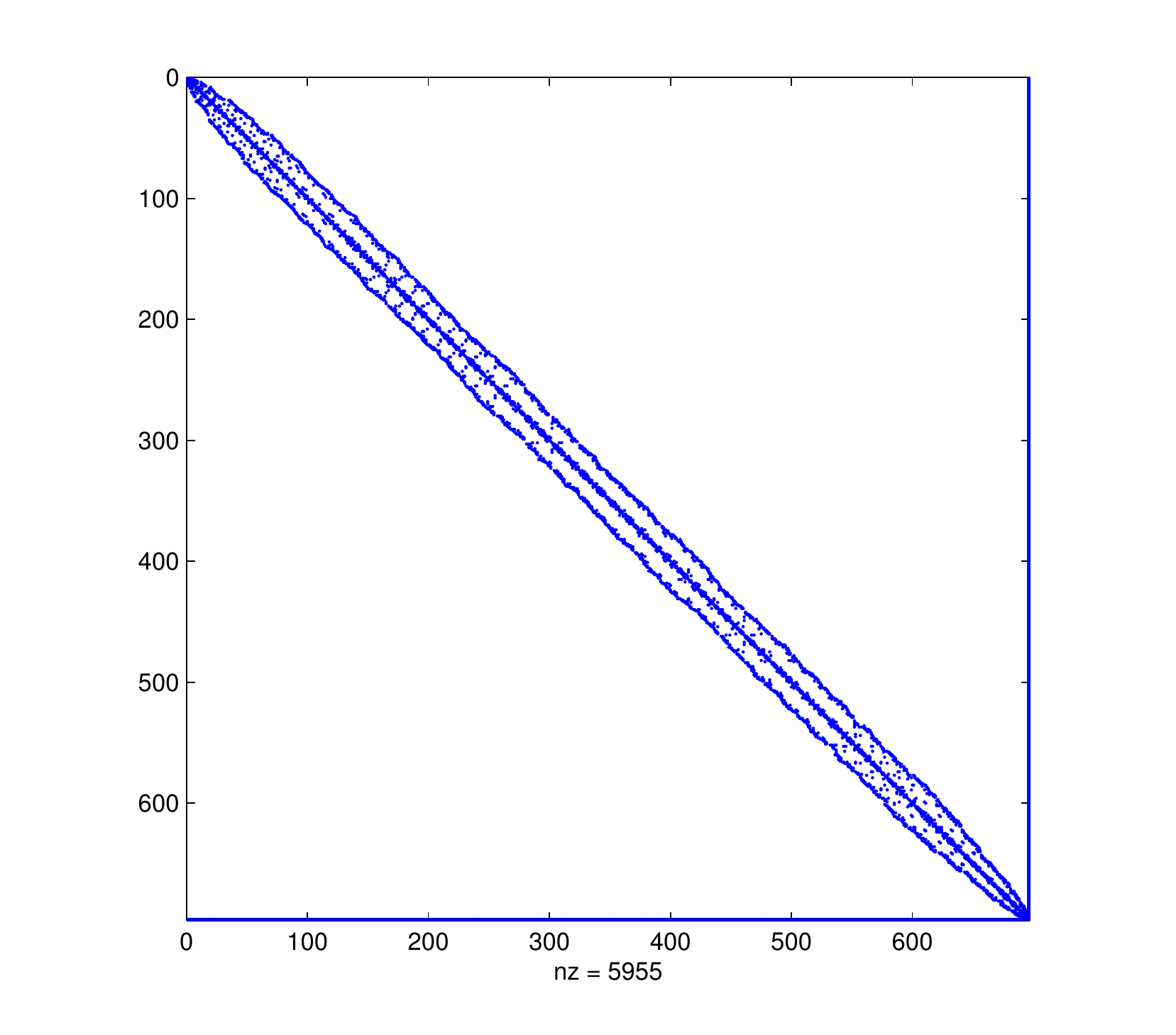}
\caption{Example sparsity plots for $M$ in Poisson problem with square domain (left) and  ring-like domain (right).}
\label{fig:lapspy}
\end{figure}

Results of both the standard solve and null space method are presented in Table \ref{tab:t3} in Appendix \ref{sec:Apoisson}.   For problems of increasing size, using the null space method as a prestructuring technique produces a significant speedup over the original system.
A visualization of the performance of both the null space method and the standard solve is given in Figure \ref{fig:lap}.

To investigate how varying nonzero patterns in $A$ may influence the performance of the null space method, the problem was also solved on a ring-like domain that consists of a circle with an elliptic section removed from the center.  The null space method and the standard solve are compared in Tables \ref{tab:tring} in Appendix \ref{sec:Apoisson}.   In this case, the inflation induced by the null space method is significantly higher than the square domain, almost doubling the number of nonzero entries in $M$.  For this mesh, the number of neighboring elements is generally higher than that of the square mesh, as $A$ averages around 8 nonzero entries per row for the ring-like domain, compared to an average of around 6 nonzero entries per row for the square domain.  The problem on the square domain also enjoys a more uniform nonzero pattern than the ring-like domain.  Overall, prestructuring the system on the ring-like domain with the null space method does not give much of an improvement over the direct solve of the original system.  
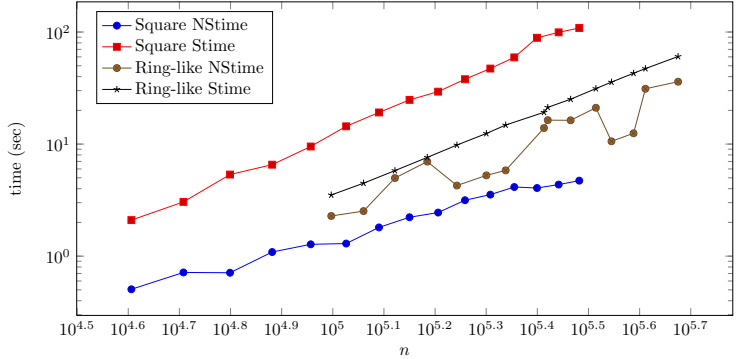
\begin{figure}
\centering
\begin{tikzpicture}[scale=0.65]
\begin{axis}[
width=15cm, height=8cm,
    xmode=log,ymode=log, xlabel={$n$}, ylabel={time (sec)},
    legend entries={Square NStime, Square Stime, Ring-like NStime, Ring-like Stime},
legend style={nodes=right},
legend pos= north west]
]
\addplot table {
40402	0.506
51077	0.715
63002	0.711
76177	1.088
90602	1.275
106277	1.295
123202	1.804
141377	2.223
160802	2.448
181477	3.155
203402	3.538
226577	4.138
251002	4.055
276677	4.351
303602	4.717
};
\addplot table {
40402	2.097
51077	3.048
63002	5.347
76177	6.551
90602	9.513
106277	14.402
123202	19.159
141377	24.801
160802	29.365
181477	37.912
203402	47.218
226577	59.386
251002	88.812
276677	99.339
303602	108.684
};
\addplot table{
99384	2.281799111
114923	2.52331192
132370	4.977851894
153152	6.976833422
175003	4.265754926
199825	5.254881852
217959	5.826424433
259022	13.89963403
263442	16.35098624
292019	16.30121916
327209	21.092035
350997	10.58647772
387725	12.4873793
408972	31.14990252
473786	36.03698308
};
\addplot table{
99384	3.514439111
114923	4.475827227
132370	5.796822507
153152	7.603971048
175003	9.804890584
199825	12.42984755
217959	14.78779719
259022	19.21815157
263442	21.2622902
292019	25.13389577
327209	31.27895974
350997	35.73339423
387725	42.86985074
408972	47.19543914
473786	60.44770178
};
\end{axis}
\end{tikzpicture}
\caption{Log-scale plot of running times for the null space method and the standard direct solve for the pure Neumann problem on a square domain and ring-like domain.}
\label{fig:lap}
\end{figure}

\subsubsection{Mixed Stokes and Navier-Stokes Problems}
The Navier-Stokes equations model the flow of an incompressible fluid in a spatial domain.  The problem with pure Dirichlet conditions is: given $f\in(L^2(\Omega))^d$, find $(u,p)$ satisfying
\begin{alignat}{2}
-\dv\left(\nu(\nabla u+\nabla u^T)\right)+(u\cdot\nabla)u+\nabla p &= f \quad&\text{ in }\Omega,
\nonumber\\
\dv(u) &= 0 \quad&\text{ in }\Omega,\label{eq:nse1} \\
u &= u_\Gamma \quad&\text{ on } \Gamma. \nonumber
\end{alignat}
Here $u$ is velocity, $p$ is pressure, and $\nu\ge 0$ is the kinematic viscosity.   Let $L^2(\Omega)_0$ represent the functions in $L^2(\Omega)$ with zero average.  The classical mixed finite element method for \eqref{eq:nse1} poses the following variational problem: find $(u,p)\in (H^1(\Omega))^d\times L^2_0(\Omega)$ such that
\begin{alignat}{2}
\nu (\nabla u,\nabla v)+((u\cdot\nabla)u,v)- (p,\dv(v)) &= (f,v) \quad&\forall v\in (H^1(\Omega))^d,
\label{eq:nse2a}\\
(\dv(u),q) &= 0 \quad&\forall q\in L^2_0(\Omega).\label{eq:nse2b} 
\end{alignat}
The Stokes problem is represented by \eqref{eq:nse1} without the nonlinear $(u\cdot\nabla)u$ term.  To implement the mean zero constraint on $p$ and $q$, a scalar Lagrange multiplier $\lambda$ can be introduced, and \eqref{eq:nse2a}--\eqref{eq:nse2b} is augmented by adding $\lambda\int_\Omega q\,d\Omega$ to the left hand side in \eqref{eq:nse2b} and the equation $\mu\int_\Omega p\,d\Omega=0$.  

Table \ref{tab:stokes} gives the results of the  Stokes problem using the standard Taylor-Hood finite elements  for velocity and pressure on a rectangular domain \cite{gir861}.   This results in a dense $B$ with around $0.11n$ nonzero entries.  Results obtained for the same computations for the fully nonlinear Navier-Stokes equations are displayed in Table \ref{tab:ns}.  An absence of data in the ``diff'' and ``Stime'' columns indicates that the standard direct solve was not attempted.  In Table \ref{tab:stokes} it is observed that the null space method can solve a system with $n+1=736030$ in less time than  the standard direct solve in a system around one fourth of the size ($n+1=183076$).  An example of the sparsity pattern of both $M$ and $Z^TAZ$ for the Navier-Stokes problem
is given in Figure \ref{fig:stokes} and the comparison of times is given in Figure \ref{fig:ns}.  While the null space method does not in general demonstrate as much speedup for Navier-Stokes as it does Stokes, it is important to note that several Navier-Stokes systems must be solved during an iterative  nonlinear solve, so the speedup will be repeatedly realized.

\begin{figure}
\centering
\includegraphics[scale=0.3]{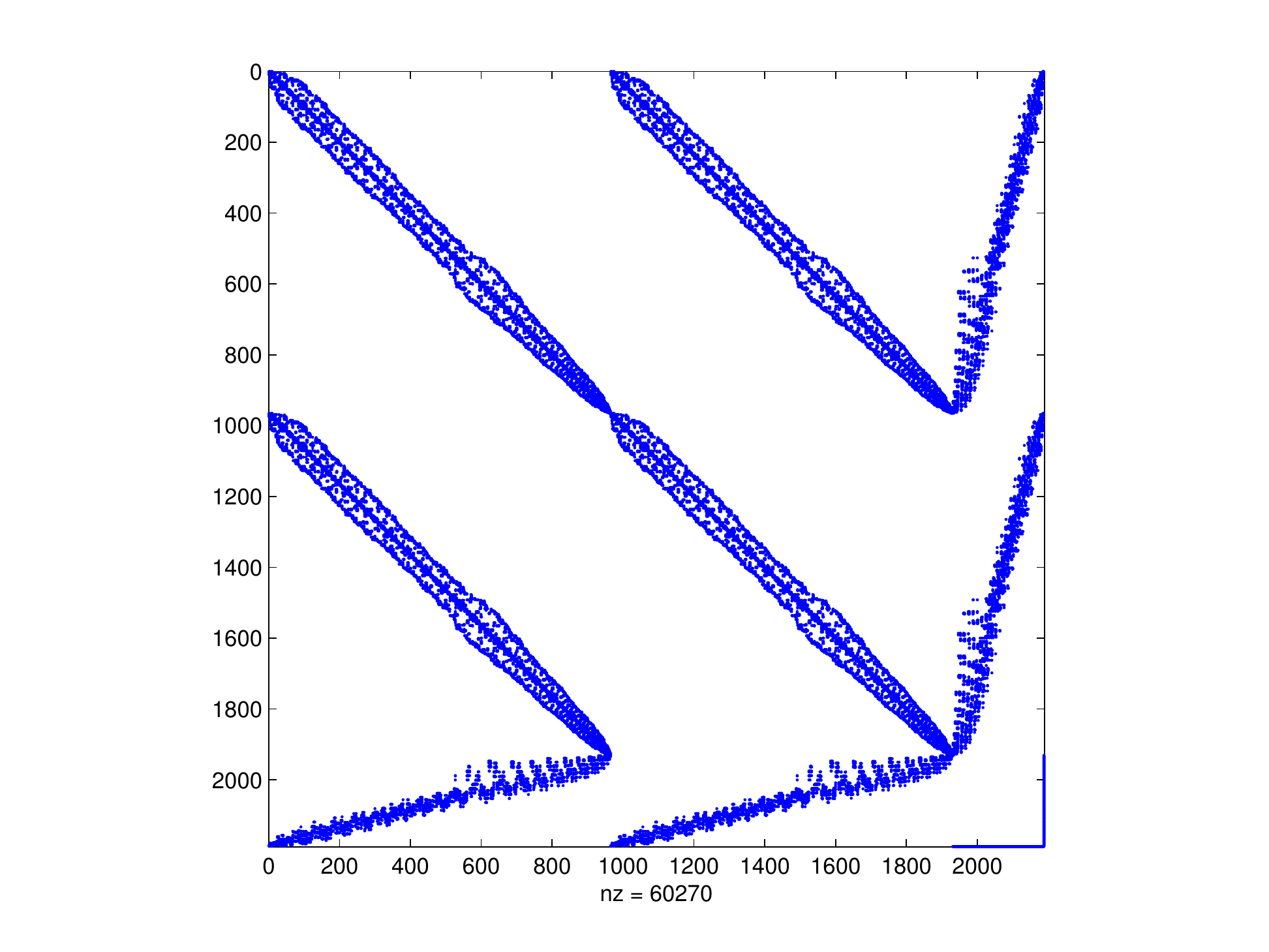}
\includegraphics[scale=0.3]{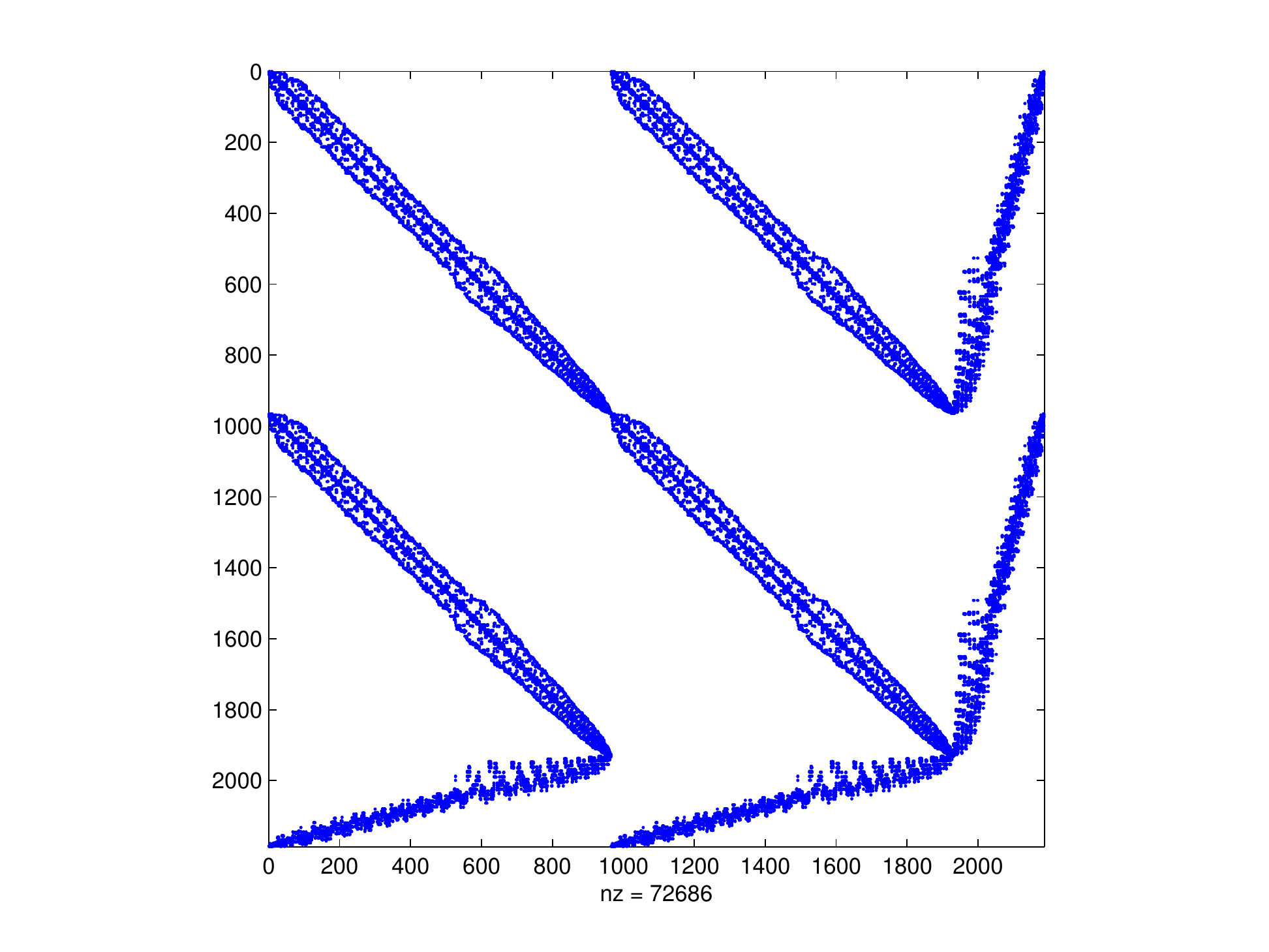}
\caption{Example sparsity plot of $M$ (left) and $Z^TAZ$ (right) for the Navier-Stokes problem.}
\label{fig:stokes}
\end{figure}

\begin{figure}
\centering
\begin{tikzpicture}[scale=0.65]
\begin{axis}[
width=15cm, height=9cm,
    xmode=log,ymode=log, xlabel={$n$}, ylabel={time (sec)},
    legend entries={Stokes NStime, Stokes Stime, Navier-Stokes NStime, Navier-Stokes Stime},
legend style={nodes=right},
legend pos= north west]
]
\addplot table {
2188	0.033
4729	0.084
8305	0.229
12772	0.369
18517	0.652
26854	0.996
32995	1.235
43078	1.866
51460	2.551
62722	3.498
73372	4.545
86866	6.509
99460	6.811
117391	9.315
132352	11.119
145864	12.530
183076	17.807
207226	22.549
221359	25.953
243286	64.222
322522	62.839
358489	132.936
374179	74.996
394189	72.277
466663	118.439
527122	136.858
590380	144.462
649570	211.806
736030	279.488
};
\addplot table {
2188	0.040
4729	0.084
8305	0.333
12772	1.474
18517	2.802
26854	6.770
32995	9.262
43078	15.819
51460	26.073
62722	42.820
73372	41.384
86866	77.948
99460	74.425
117391	127.044
132352	116.803
145864	193.427
183076	375.890
207226	411.158
221359	360.559
243286	377.703
322522	1206.523
358489	1296.813017
374179	1292.916187
394189	1226.410117
};
\addplot table {
2188	0.039
4729	0.092
8305	0.201
12772	0.312
18517	0.560
26854	0.952
32995	1.126
43078	1.679
51460	2.476
62722	2.863
73372	4.036
86866	5.286
99460	11.722
117391	8.681
132352	10.859
145864	12.294
183076	38.428
207226	21.910
221359	47.819
243286	30.516
322522	124.250
358489	126.839
394189	136.018
466663	268.038
527122	147.946
590380	324.469
649570	448.260
};
\addplot table {
2188	0.027
4729	0.104
8305	0.244
12772	0.703
18517	1.151
26854	2.327
32995	9.156
43078	8.072
51460	13.421
62722	18.841
73372	27.165
86866	63.046
99460	64.718
117391	109.731
132352	145.626
145864	144.746
183076	206.909
207226	1089.682
221359	310.419
243286	530.607
322522	2475.251
358489	958.905
394189	1357.513
};
\end{axis}
\end{tikzpicture}
\caption{Log-scale plot of running times for the null space method and the standard direct solve for Stokes and Navier-Stokes.}
\label{fig:ns}
\end{figure}
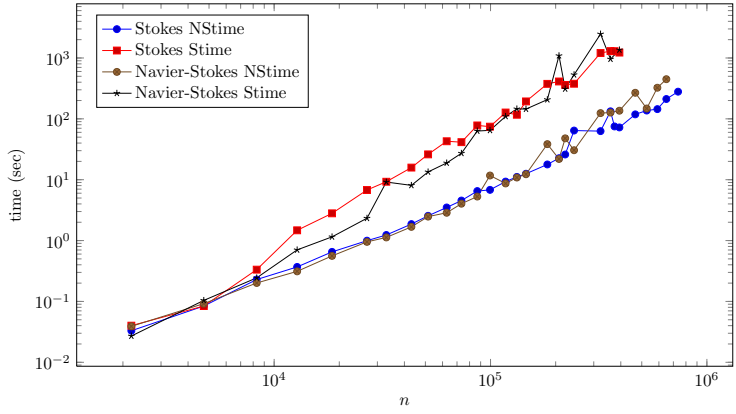

\subsubsection{Dual-Mixed Stokes  Problems}
While the velocity and pressure are the primary unknowns in the classical mixed method for the Stokes and Navier-Stokes equations, the dual-mixed method \cite{how095,wal091,wal101} for Stokes and Navier-Stokes approximates the fluid stress ($S$), velocity ($u$), and velocity gradient ($G$) directly.  Specifically, the variational problem is: given $f\in (L^2(\Omega))^d$ and $g\in (H^{1/2}(\Gamma))^d$, find $(G,u,S)\in (L^2(\Omega))^{d\times d}\times (L^2(\Omega))^{d}\times \mathbb{S}$ satisfying
\begin{align}
  \nu(G,H) - (1/2)(u\otimes u, H) - (S, H) 
  & = 0, & \forall H  \in (L^2(\Omega))^{d\times d}, \nonumber \\ 
  (1/2)(G u, v) - (\dv(S), v) 
  & = (f, v), & \forall v  \in (L^2(\Omega))^{d}, \label{eqn:dmNS} \\ 
  (G,T) + (u, \dv(T)) 
  & = \int_\Gamma u_\Gamma\cdot Tn\,d\Gamma, & \forall T  \in \mathbb{S}, \nonumber
\end{align}
where 
\[\mathbb{S} = \left\{T\in (H(\dv,\Omega))^{d\times d}\sst \int_\Omega tr(T)\,d\Omega=0\right\}.\]
The mean trace zero constraint on $(H(\dv,\Omega))^{d\times d}$ can be relaxed by incorporating a scalar Lagrange multiplier $\lambda\in\real$, yielding the linear system 
 \begin{equation}\begin{bmatrix}
A_G & A_{Gu} & B_1^T & 0\\
A_{uG} & 0&  B_2^T & 0\\
B_1& B_2 & 0 &C^T\\
0 & 0&C&0
\end{bmatrix}
\begin{bmatrix}
G\\u\\S\\\lambda
\end{bmatrix}=\begin{bmatrix}
0\\F\\g\\0
\end{bmatrix},
\label{eq:dmstokesmat}
\end{equation}
where $C$ is the dense row that represents $\int_\Omega tr(S)\,d\Omega$.
Linear systems for the dual-mixed Stokes problem (without the $A_{Gu}$ and $A_{uG}$ blocks in \eqref{eq:dmstokesmat}) were constructed for two different spatial domains in $\real^2$: a unit square and an irregular-shaped domain.  First-order Raviart-Thomas elements for $S$ and discontinuous piecewise linear elements for $G$ and $u$ were employed, resulting in $|B|\approx 0.4n$.   Comparisons of the sparsity plots for the two different domains are presented in Figure \ref{fig:dmstokes}.  Results from the full null space method and the standard direct solve are given in Tables \ref{tab:dmstokes} and \ref{tab:dmstokesa}, and the comparisons are summarized in Figure \ref{fig:dmstokesa}.  The speedup exhibited by the null space method is even greater for the irregular domain, suggesting that the removal of the dense row and column allows the algorithms in the direct solver to take a greater advantage of the lack of structure in the irregular domain problem.  This can be observed by examining the supernodal column elimination tree \cite{demmel99} (obtained using the symbolic factorization option of the \verb|umfpack2| routine provided by SuiteSparse \cite{davisSS}) of both $M$ and $Z^TAZ$ given for $n+1=139616$ in Figure \ref{fig:tree}.  Prestructuring $M$ with the null space method effectively reduces the height of the tree from 23220 to 98, vastly reducing the intercolumn dependencies.

\begin{figure}
\centering
\includegraphics[scale=0.25]{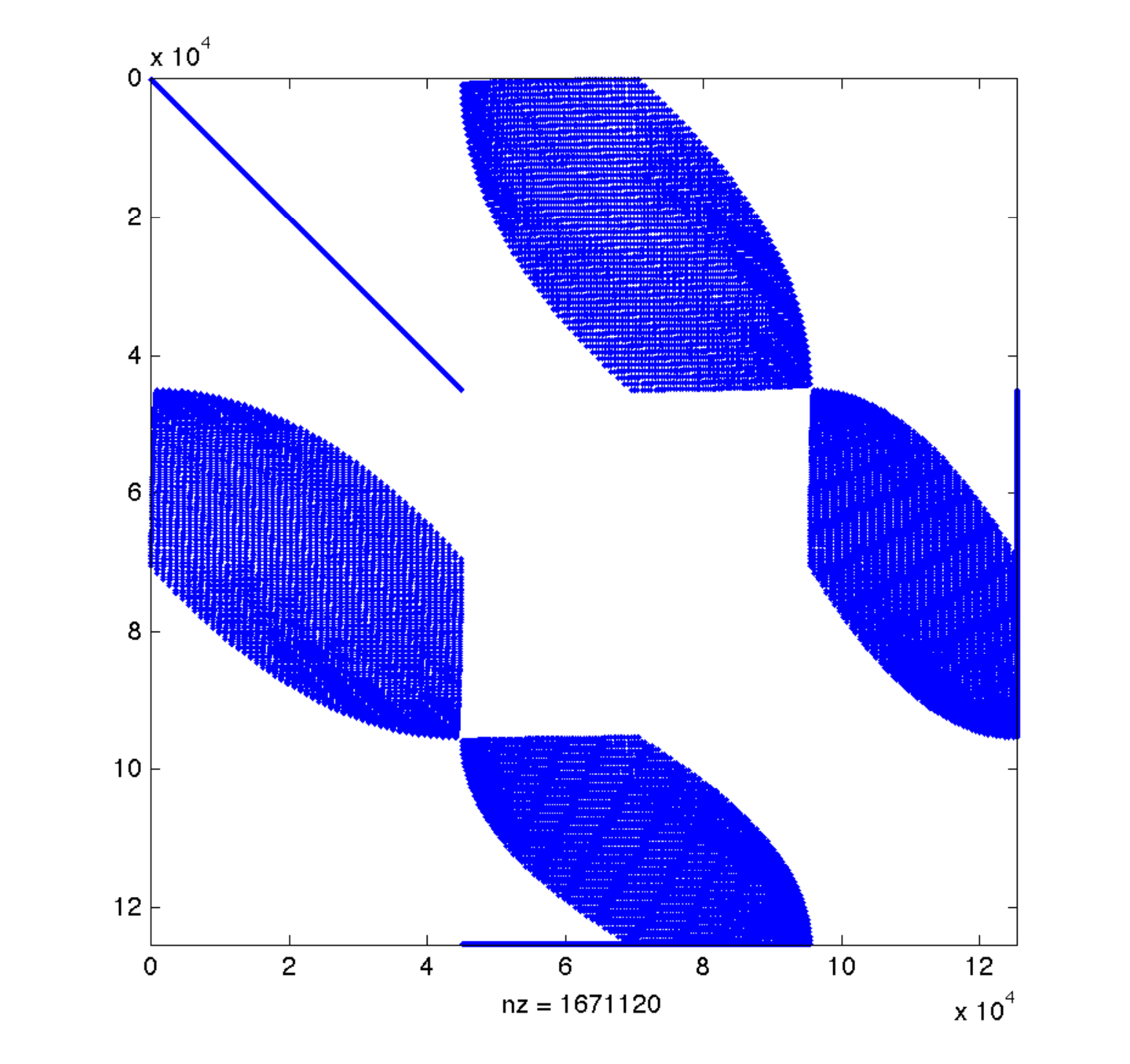}
\hskip 0.1in
\includegraphics[scale=0.25]{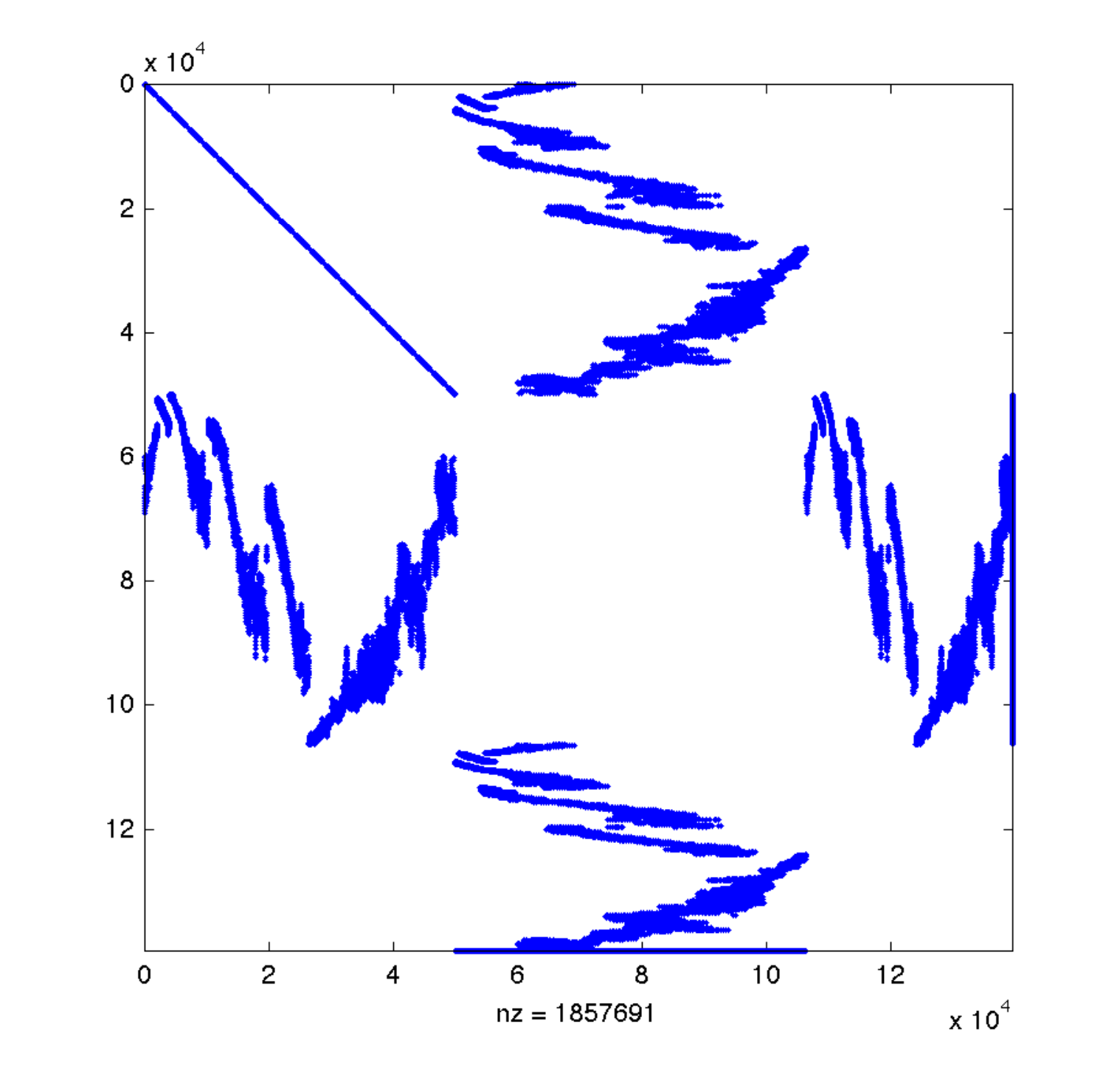}
\hskip 0.1in
\includegraphics[scale=0.2]{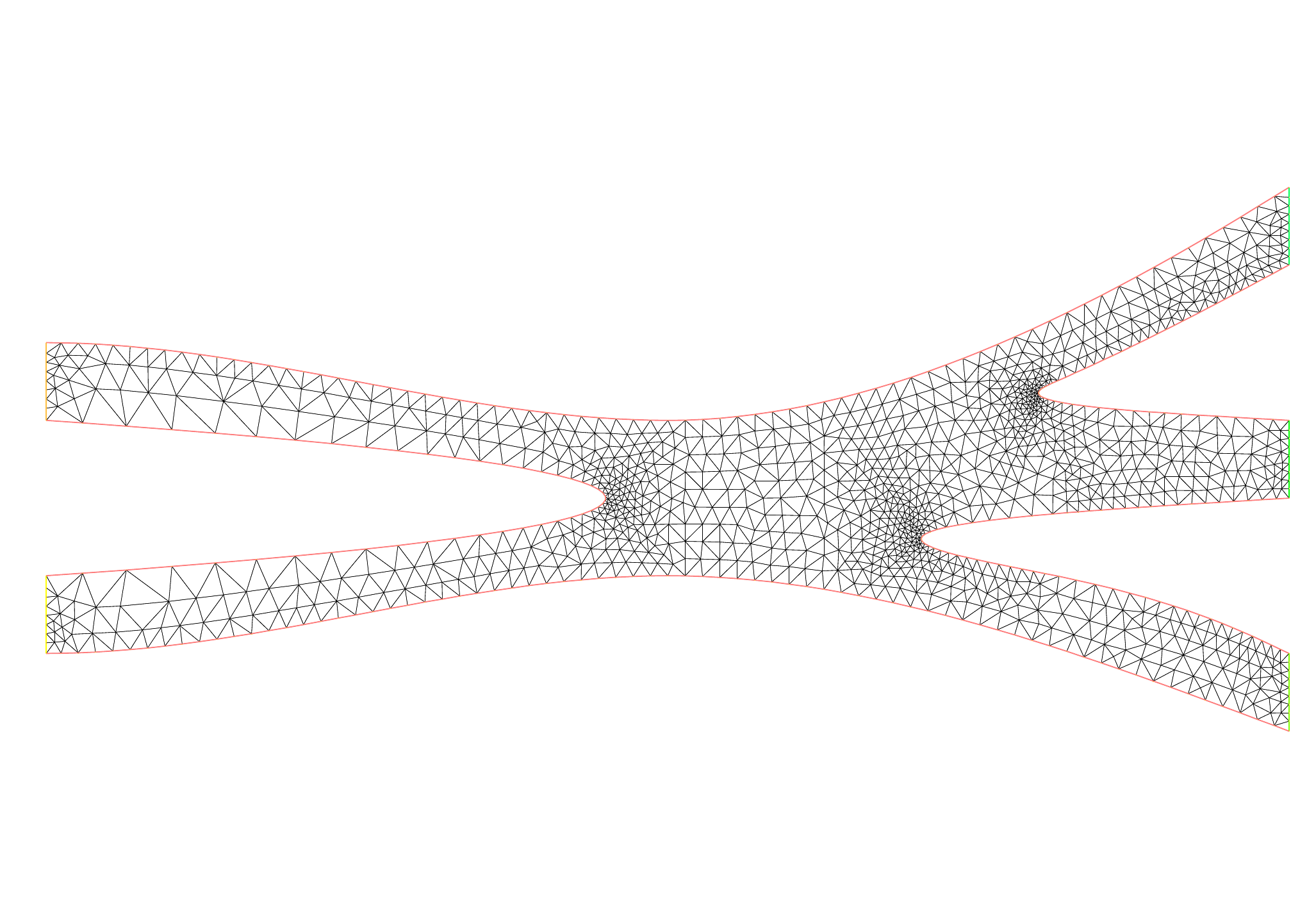}
\caption{Example sparsity plots of dual-mixed stokes system on square domain (left) and irregular-shaped domain (center), and irregular-shaped domain mesh (right).}
\label{fig:dmstokes}
\end{figure}
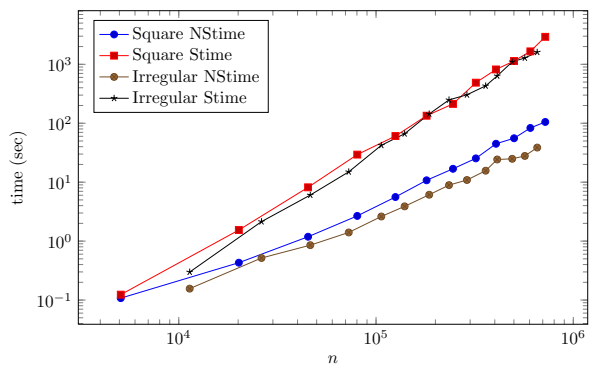
\begin{figure}
\centering
\begin{tikzpicture}[scale=0.65]
\begin{axis}[
width=12cm, height=8cm,
    xmode=log,ymode=log, xlabel={$n$}, ylabel={time (sec)},
    legend entries={Square NStime, Square Stime, Irregular NStime, Irregular Stime},
legend style={nodes=right},
legend pos= north west]
]
\addplot table {
5081	0.107784793
20161	0.430277864
45241	1.186198445
80321	2.675769318
125401	5.599631403
180481	10.75938984
245561	16.8670154
320641	25.34948645
405721	44.85087991
500801	55.61781144
605881	83.16007014
720961	105.1580089
};
\addplot table {
5081	0.123519295
20161	1.541800648
45241	8.180284204
80321	29.36862138
125401	60.64343569
180481	134.2625852
245561	213.088345
320641	488.0833347
405721	818.0186287
500801	1138.316574
605881	1657.527062
720961	2926.809731
};
\addplot table {
11359	0.156102172
26265	0.516805124
46421	0.85362569
72727	1.402187337
106360	2.617292788
139616	3.8888757
186072	6.123803578
234428	8.914583328
288984	10.88550501
359490	15.65494506
411446	24.27388225
490102	24.91889986
567858	27.87281518
656314	38.66237248
};
\addplot table {
11359	0.298645989
26265	2.140558437
46421	6.046420909
72727	14.9809734
106360	42.1990819
139616	66.35144904
186072	143.8916848
234428	248.9575103
288984	303.9236028
359490	430.5084797
411446	635.7332143
490102	1099.456987
567858	1281.731702
656314	1604.043034
};
\end{axis}
\end{tikzpicture}
\caption{Log-scale plot of running times for the null space method and the standard direct solve for dual-mixed Stokes on square and irregular domains.}
\label{fig:dmstokesa}
\end{figure}
\begin{figure}
\centering
\includegraphics[scale=0.25]{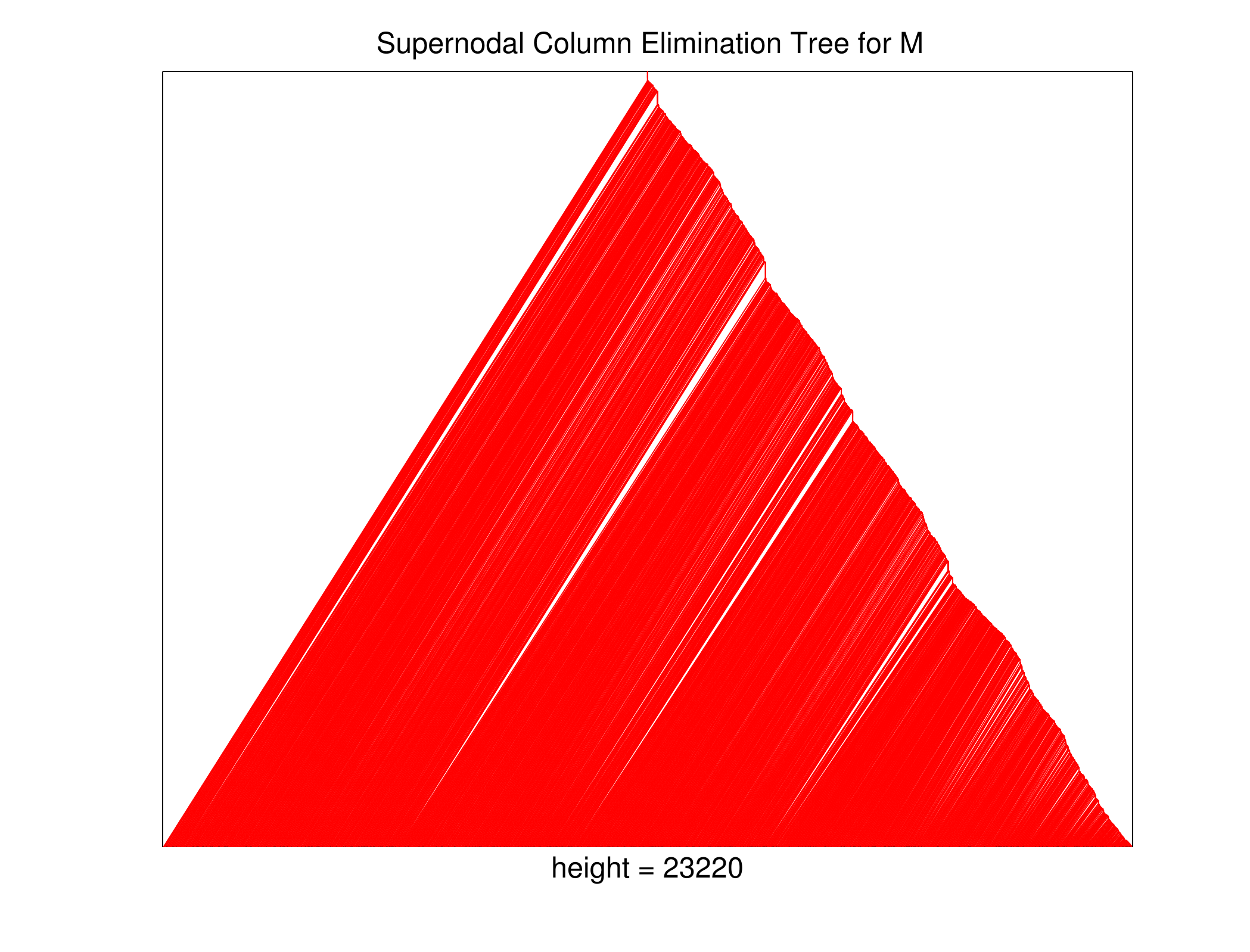}
\hskip 0.25in
\includegraphics[scale=0.25]{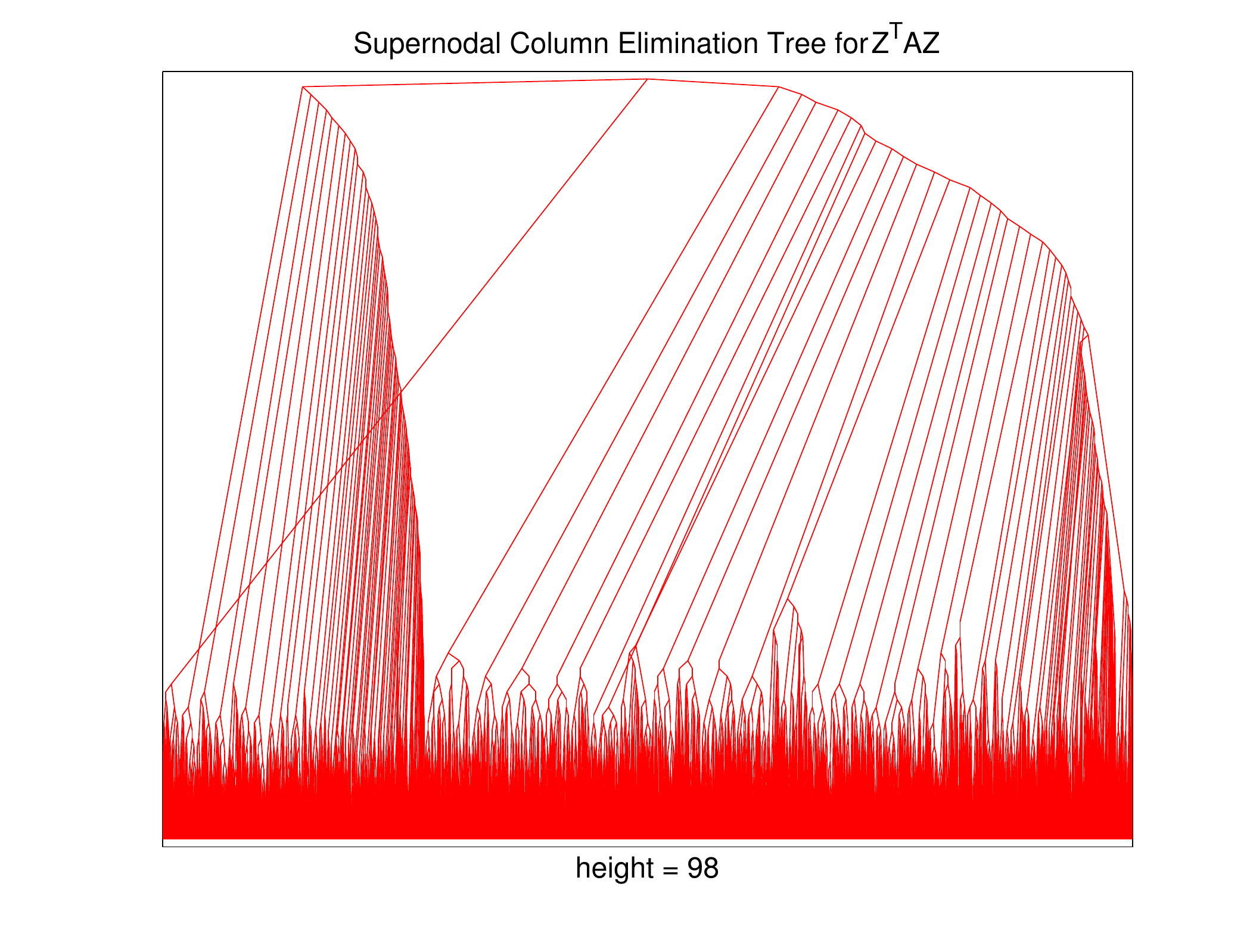}
\caption{Plots of supernodal column elimination tree for $M$ (left) and $Z^TAZ$ (right) for the dual-mixed stokes problem on irregular-shaped domain with $n+1=139616$.}
\label{fig:tree}
\end{figure}

\section{Summary}\label{sec:sum}
The null space methods presented here can provide end users of direct solvers a means to avoid decreases in solver performance primarily due to the presence of dense rows.  The methods  are easy to implement, and the construction of the null space basis can be done in a fraction of the time required by the linear solver.  Several examples demonstrate significant speedup in solving the prestructured system over the original system, however the performance of the methods appears to be limited when the inflation is significant.  
 More investigation is required to determine under what conditions systems with the two-sided or one-sided null space method are more likely to be beneficial.
While improvements in direct solver performance are continually advancing the state of the art, the techniques described here will currently allow end users of direct solvers to solve a larger class of problems with existing resources.

\appendix
\section{Tables For Arrowhead Matrices}\label{sec:Aarrow}
\begin{table}[H]\centering
\begin{small}
\begin{tabular}{r|r|r|r|r|r|r|r|r}
$n+1$	&	$|M|$	&	$|Z_1^TAZ_2|$	&	infl	&	diff	&	Ztime	&	NStime	&	Stime	&	speedup	\\\hline
25001	&	75001	&	74998	&	1.00	&	3.357E-13	&	0.013	&	0.070	&	0.392	&	5.63	\\
50001	&	150001	&	149998	&	1.00	&	3.387E-12	&	0.018	&	0.125	&	1.499	&	11.99	\\
75001	&	225001	&	224998	&	1.00	&	2.649E-12	&	0.027	&	0.199	&	3.318	&	16.67	\\
100001	&	300001	&	299998	&	1.00	&	2.140E-12	&	0.037	&	0.271	&	5.858	&	21.60	\\
125001	&	375001	&	374998	&	1.00	&	1.008E-11	&	0.048	&	0.339	&	9.066	&	26.74	\\
150001	&	450001	&	449998	&	1.00	&	8.477E-12	&	0.056	&	0.410	&	12.999	&	31.67	\\
175001	&	525001	&	524998	&	1.00	&	3.956E-12	&	0.065	&	0.467	&	17.649	&	37.80	\\
200001	&	600001	&	599998	&	1.00	&	3.743E-12	&	0.077	&	0.568	&	23.021	&	40.53	\\
225001	&	675001	&	674998	&	1.00	&	4.426E-12	&	0.089	&	0.631	&	29.107	&	46.13	\\
250001	&	750001	&	749998	&	1.00	&	1.796E-11	&	0.100	&	0.688	&	35.874	&	52.16	\\
275001	&	825001	&	824998	&	1.00	&	2.039E-11	&	0.114	&	0.752	&	43.379	&	57.70	\\
300001	&	900001	&	899998	&	1.00	&	2.273E-12	&	0.125	&	0.855	&	51.612	&	60.36	\\
325001	&	975001	&	974998	&	1.00	&	1.068E-11	&	0.138	&	0.914	&	60.450	&	66.15	\\
350001	&	1050001	&	1049998	&	1.00	&	3.340E-11	&	0.151	&	0.999	&	70.263	&	70.34	\\
375001	&	1125001	&	1124998	&	1.00	&	3.683E-12	&	0.161	&	1.072	&	80.487	&	75.05	\\
400001	&	1200001	&	1199998	&	1.00	&	2.944E-11	&	0.173	&	1.138	&	91.617	&	80.54	\\
425001	&	1275001	&	1274998	&	1.00	&	4.299E-12	&	0.193	&	1.223	&	103.670	&	84.78	\\
450001	&	1350001	&	1349998	&	1.00	&	6.231E-11	&	0.199	&	1.276	&	116.110	&	91.03	\\
475001	&	1425001	&	1424998	&	1.00	&	3.010E-12	&	0.214	&	1.364	&	129.052	&	94.62	\\
500001	&	1500001	&	1499998	&	1.00	&	3.455E-11	&	0.224	&	1.427	&	143.563	&	100.57	\\
\end{tabular}
\end{small}
\caption{Comparison of standard direct solve and one-sided null space method for arrowhead matrices for increasing $n$ and $|B_1|=n$.}
\label{tab:t2}
\end{table}

\begin{table}\centering
\begin{small}
\begin{tabular}{r|r|r|r|r|r|r|r|r}
$|M|$	&	$|B|$	&	$|Z^TAZ|$	&	infl	&	diff	&	Ztime	&	NStime	&	Stime	&	speedup	\\\hline
250007	&	4	&	250004	&	1.00	&	2.79E-16	&	0.0085	&	0.104	&	0.269	&	2.60	\\
250051	&	26	&	250048	&	1.00	&	1.44E-15	&	0.0079	&	0.111	&	0.292	&	2.63	\\
250501	&	251	&	250498	&	1.00	&	1.04E-13	&	0.0079	&	0.107	&	0.278	&	2.59	\\
254979	&	2490	&	254976	&	1.00	&	7.64E-13	&	0.0080	&	0.114	&	0.297	&	2.61	\\
262349	&	6175	&	262346	&	1.00	&	7.35E-12	&	0.0101	&	0.122	&	0.344	&	2.82	\\
274437	&	12219	&	274434	&	1.00	&	1.68E-12	&	0.0120	&	0.139	&	0.478	&	3.45	\\
297581	&	23791	&	297578	&	1.00	&	4.27E-10	&	0.0245	&	0.165	&	0.965	&	5.86	\\
340627	&	45314	&	340624	&	1.00	&	6.71E-11	&	0.0330	&	0.219	&	2.312	&	10.58	\\
414717	&	82359	&	414714	&	1.00	&	2.29E-10	&	0.0449	&	0.284	&	5.621	&	19.82	\\
446653	&	98327	&	446650	&	1.00	&	5.93E-11	&	0.0508	&	0.353	&	7.558	&	21.44	\\
475367	&	112684	&	475364	&	1.00	&	2.16E-10	&	0.0567	&	0.373	&	9.339	&	25.05	\\
501603	&	125802	&	501600	&	1.00	&	1.20E-09	&	0.0593	&	0.440	&	11.422	&	25.97	\\
525071	&	137536	&	525068	&	1.00	&	5.00E-10	&	0.0649	&	0.448	&	13.299	&	29.70	\\
566043	&	158022	&	566040	&	1.00	&	2.95E-09	&	0.0725	&	0.524	&	16.871	&	32.17	\\
750001	&	250001	&	749998	&	1.00	&	8.54E-10	&	0.1028	&	0.722	&	38.479	&	53.32	
\end{tabular}
\end{small}
\caption{Comparison of standard direct solve and one-sided null space method for arrowhead matrices, $n=250000$ and increasing $|B|$.}
\label{tab:arrow2}
\end{table}

\section{Tables for Poisson Problem}\label{sec:Apoisson}

\begin{table}[H]\centering
\begin{small}
\begin{tabular}{r|r|r|r|r|r|r|r|r}
$n+1$	&	$|M|$	&	$|Z^TAZ|$	&	infl	&	diff	&	Ztime	&	NStime	&	Stime	&	speedup	\\\hline
40402	&	282003	&	442788	&	1.57	&	1.88E-12	&	0.025	&	0.506	&	2.097	&	4.14	\\
51077	&	356628	&	560013	&	1.57	&	7.02E-12	&	0.022	&	0.715	&	3.048	&	4.26	\\
63002	&	440003	&	690988	&	1.57	&	7.39E-12	&	0.035	&	0.711	&	5.347	&	7.52	\\
76177	&	532128	&	835713	&	1.57	&	8.49E-12	&	0.032	&	1.088	&	6.551	&	6.02	\\
90602	&	633003	&	994188	&	1.57	&	9.03E-11	&	0.038	&	1.275	&	9.513	&	7.46	\\
106277	&	742628	&	1166413	&	1.57	&	7.63E-11	&	0.050	&	1.295	&	14.402	&	11.12	\\
123202	&	861003	&	1352388	&	1.57	&	2.00E-10	&	0.060	&	1.804	&	19.159	&	10.62	\\
141377	&	988128	&	1552113	&	1.57	&	3.11E-10	&	0.071	&	2.223	&	24.801	&	11.16	\\
160802	&	1124003	&	1765588	&	1.57	&	4.59E-10	&	0.081	&	2.448	&	29.365	&	12.00	\\
181477	&	1268628	&	1992813	&	1.57	&	5.90E-10	&	0.084	&	3.155	&	37.912	&	12.02	\\
203402	&	1422003	&	2233788	&	1.57	&	3.61E-10	&	0.096	&	3.538	&	47.218	&	13.35	\\
226577	&	1584128	&	2488513	&	1.57	&	2.85E-10	&	0.104	&	4.138	&	59.386	&	14.35	\\
251002	&	1755003	&	2756988	&	1.57	&	6.12E-10	&	0.119	&	4.055	&	88.812	&	21.90	\\
276677	&	1934628	&	3039213	&	1.57	&	5.25E-09	&	0.131	&	4.351	&	99.339	&	22.83	\\
303602	&	2123003	&	3335188	&	1.57	&	3.92E-09	&	0.139	&	4.717	&	108.684	&	23.04	
\end{tabular}
\end{small}
\caption{Comparison of standard direct solve and null space method for pure Neumann problem on square domain.}
\label{tab:t3}
\end{table}

\begin{table}[H]\centering
\begin{small}
\begin{tabular}{r|r|r|r|r|r|r|r|r}
$n+1$	&	$|M|$	&	$|Z^TAZ|$	&	infl	&	diff	&	Ztime	&	NStime	&	Stime	&	speedup	\\\hline
99384	&	890697	&	1704464	&	1.91	&	1.13E-09	&	0.035	&	2.282	&	3.514	&	1.54	\\
114923	&	1030248	&	1964071	&	1.91	&	3.90E-09	&	0.040	&	2.523	&	4.476	&	1.77	\\
132370	&	1186971	&	2368424	&	2.00	&	1.46E-09	&	0.048	&	4.978	&	5.797	&	1.16	\\
153152	&	1373709	&	2771042	&	2.02	&	3.49E-09	&	0.055	&	6.977	&	7.604	&	1.09	\\
175003	&	1570068	&	3003661	&	1.91	&	8.42E-09	&	0.064	&	4.266	&	9.805	&	2.30	\\
199825	&	1793166	&	3446905	&	1.92	&	1.60E-09	&	0.073	&	5.255	&	12.430	&	2.37	\\
217959	&	1956072	&	3753307	&	1.92	&	2.18E-08	&	0.081	&	5.826	&	14.788	&	2.54	\\
259022	&	2325325	&	4482611	&	1.93	&	2.79E-08	&	0.098	&	13.900	&	19.218	&	1.38	\\
263442	&	2364819	&	4756114	&	2.01	&	3.78E-08	&	0.104	&	16.351	&	21.262	&	1.30	\\
292019	&	2621712	&	5146773	&	1.96	&	1.17E-08	&	0.108	&	16.301	&	25.134	&	1.54	\\
327209	&	2938122	&	5844671	&	1.99	&	7.83E-10	&	0.124	&	21.092	&	31.279	&	1.48	\\
350997	&	3151912	&	6041955	&	1.92	&	3.96E-08	&	0.138	&	10.586	&	35.733	&	3.38	\\
387725	&	3482160	&	6644983	&	1.91	&	5.75E-08	&	0.152	&	12.487	&	42.870	&	3.43	\\
408972	&	3673087	&	7302360	&	1.99	&	7.88E-09	&	0.164	&	31.150	&	47.195	&	1.52	\\
473786	&	4256110	&	8353485	&	1.96	&	1.24E-08	&	0.194	&	36.037	&	60.448	&	1.68	\\
\end{tabular}
\end{small}
\caption{Comparison of standard direct solve and null space method for pure Neumann problem on ring-like domain.}
\label{tab:tring}
\end{table}

\section{Tables for Mixed Stokes and Navier-Stokes}\label{sec:Amixed}

\begin{table}[H]\centering
\begin{footnotesize}
\begin{tabular}{r|r|r|r|r|r|r|r|r|r}
$n+1$	&	$|M|$	&	$|B|$	&	$|Z^TAZ|$	&	infl	&	diff	&	Ztime	&	NStime	&	Stime	&	speedup	\\\hline
2188	&	38576	&	257	&	50992	&	1.32	&	3.72E-15	&	0.0003	&	0.033	&	0.040	&	1.20	\\
4729	&	84820	&	546	&	112226	&	1.32	&	4.14E-15	&	0.0004	&	0.084	&	0.084	&	1.00	\\
8305	&	149502	&	950	&	198222	&	1.33	&	4.97E-15	&	0.0005	&	0.229	&	0.333	&	1.46	\\
12772	&	231004	&	1453	&	305296	&	1.32	&	9.84E-15	&	0.0008	&	0.369	&	1.474	&	3.99	\\
18517	&	337682	&	2098	&	447860	&	1.33	&	1.56E-14	&	0.0016	&	0.652	&	2.802	&	4.30	\\
26854	&	489294	&	3031	&	646072	&	1.32	&	1.38E-14	&	0.0018	&	0.996	&	6.770	&	6.79	\\
32995	&	602306	&	3720	&	796240	&	1.32	&	2.67E-14	&	0.0023	&	1.235	&	9.262	&	7.50	\\
43078	&	786882	&	4847	&	1036976	&	1.32	&	4.76E-14	&	0.0032	&	1.866	&	15.819	&	8.48	\\
51460	&	941098	&	5785	&	1246734	&	1.32	&	1.82E-14	&	0.0040	&	2.551	&	26.073	&	10.22	\\
62722	&	1150150	&	7043	&	1523614	&	1.32	&	1.77E-14	&	0.0039	&	3.498	&	42.820	&	12.24	\\
73372	&	1346148	&	8233	&	1795402	&	1.33	&	2.36E-14	&	0.0050	&	4.545	&	41.384	&	9.11	\\
86866	&	1593288	&	9739	&	2121702	&	1.33	&	3.34E-14	&	0.0052	&	6.509	&	77.948	&	11.98	\\
99460	&	1825198	&	11145	&	2431356	&	1.33	&	2.86E-14	&	0.0058	&	6.811	&	74.425	&	10.93	\\
117391	&	2151318	&	13144	&	2859160	&	1.33	&	2.97E-14	&	0.0066	&	9.315	&	127.044	&	13.64	\\
132352	&	2429104	&	14813	&	3233876	&	1.33	&	3.00E-14	&	0.0076	&	11.119	&	116.803	&	10.51	\\
145864	&	2668852	&	16321	&	3543900	&	1.33	&	4.65E-14	&	0.0103	&	12.530	&	193.427	&	15.44	\\
183076	&	3362662	&	20469	&	4473160	&	1.33	&	5.89E-14	&	0.0126	&	17.807	&	375.890	&	21.11	\\
207226	&	3807856	&	23159	&	5056400	&	1.33	&	3.51E-14	&	0.0141	&	22.549	&	411.158	&	18.23	\\
221359	&	4080710	&	24736	&	5423956	&	1.33	&	4.35E-14	&	0.0154	&	25.953	&	360.559	&	13.89	\\
243286	&	4476336	&	27179	&	5959142	&	1.33	&	4.98E-14	&	0.0173	&	64.222	&	377.703	&	5.88	\\
322522	&	5937362	&	36003	&	7907618	&	1.33	&	6.25E-14	&	0.0238	&	62.839	&	1206.523	&	19.20	\\
358489	&	6593518	&	40006	&	8741244	&	1.33	&	5.94E-14	&	0.0259	&	132.936	&	1296.813017	&	9.76	\\
374179	&	6886288	&	41756	&	9182944	&	1.33	&	8.30E-14	&	0.0280	&	74.996	&	1292.916187	&	17.24	\\
394189	&	7267594	&	43986	&	9675046	&	1.33	&	7.21E-14	&	0.0284	&	72.277	&	1226.410117	&	16.97	\\
466663	&	8584006	&	52052	&	11442590	&	1.33	&		&	0.0384	&	118.439	&		&		\\
527122	&	9703388	&	58783	&	12932988	&	1.33	&		&	0.0636	&	136.858	&		&		\\
590380	&	10854476	&	65825	&	14401620	&	1.33	&		&	0.0729	&	144.462	&		&		\\
649570	&	11977796	&	72415	&	15900372	&	1.33	&		&	0.0787	&	211.806	&		&		\\
736030	&	13549470	&	82035	&	18021144	&	1.33	&		&	0.0936	&	279.488	&		&		
\end{tabular}
\end{footnotesize}
\caption{Comparison of standard direct solve and null space method for Stokes problem.}
\label{tab:stokes}
\end{table}

\begin{table}[H]\centering
\begin{footnotesize}
\begin{tabular}{r|r|r|r|r|r|r|r|r|r}
$n+1$	&	$|M|$	&	$|B|$	&	$|Z^TAZ|$	&	infl	&	diff	&	Ztime	&	NStime	&	Stime	&	speedup	\\\hline
2188	&	60270	&	257	&	72686	&	1.21	&	1.77E-15	&	0.0003	&	0.039	&	0.027	&	0.70	\\
4729	&	132406	&	546	&	159812	&	1.21	&	5.06E-15	&	0.0004	&	0.092	&	0.104	&	1.13	\\
8305	&	233862	&	950	&	282582	&	1.21	&	8.31E-15	&	0.0005	&	0.201	&	0.244	&	1.21	\\
12772	&	361190	&	1453	&	435482	&	1.21	&	1.01E-14	&	0.0007	&	0.312	&	0.703	&	2.25	\\
18517	&	526936	&	2098	&	637114	&	1.21	&	2.10E-14	&	0.0013	&	0.560	&	1.151	&	2.05	\\
26854	&	765634	&	3031	&	922412	&	1.20	&	1.28E-14	&	0.0023	&	0.952	&	2.327	&	2.44	\\
32995	&	941480	&	3720	&	1135414	&	1.21	&	6.35E-12	&	0.0017	&	1.126	&	9.156	&	8.13	\\
43078	&	1231012	&	4847	&	1481106	&	1.20	&	1.73E-14	&	0.0026	&	1.679	&	8.072	&	4.81	\\
51460	&	1471656	&	5785	&	1777292	&	1.21	&	1.74E-14	&	0.0034	&	2.476	&	13.421	&	5.42	\\
62722	&	1797106	&	7043	&	2170570	&	1.21	&	1.84E-14	&	0.0038	&	2.863	&	18.841	&	6.58	\\
73372	&	2103260	&	8233	&	2552514	&	1.21	&	3.64E-12	&	0.0046	&	4.036	&	27.165	&	6.73	\\
86866	&	2490402	&	9739	&	3018816	&	1.21	&	3.07E-14	&	0.0054	&	5.286	&	63.046	&	11.93	\\
99460	&	2852024	&	11145	&	3458182	&	1.21	&	2.03E-10	&	0.0063	&	11.722	&	64.718	&	5.52	\\
117391	&	3366344	&	13144	&	4074186	&	1.21	&	1.57E-11	&	0.0071	&	8.681	&	109.731	&	12.64	\\
132352	&	3798722	&	14813	&	4603494	&	1.21	&	8.35E-12	&	0.0080	&	10.859	&	145.626	&	13.41	\\
145864	&	4177670	&	16321	&	5052718	&	1.21	&	3.55E-14	&	0.0098	&	12.294	&	144.746	&	11.77	\\
183076	&	5255710	&	20469	&	6366208	&	1.21	&	4.28E-14	&	0.0123	&	38.428	&	206.909	&	5.38	\\
207226	&	5956162	&	23159	&	7204706	&	1.21	&	7.68E-11	&	0.0140	&	21.910	&	1089.682	&	49.73	\\
221359	&	6370840	&	24736	&	7714086	&	1.21	&	4.76E-14	&	0.0148	&	47.819	&	310.419	&	6.49	\\
243286	&	6994304	&	27179	&	8477110	&	1.21	&	1.97E-12	&	0.0159	&	30.516	&	530.607	&	17.39	\\
322522	&	9280664	&	36003	&	11250920	&	1.21	&	5.57E-14	&	0.0223	&	124.250	&	2475.251	&	19.92	\\
358489	&	10323538	&	40006	&	12471264	&	1.21	&	6.25E-14	&	0.0253	&	126.839	&	958.905	&	7.56	\\
394189	&	11350622	&	43986	&	13758074	&	1.21	&	6.91E-14	&	0.0275	&	136.018	&	1357.513	&	9.98	\\
466663	&	13424512	&	52052	&	16283096	&	1.21	&		&	0.0343	&	268.038	&		&		\\
527122	&	15174044	&	58783	&	18403644	&	1.21	&		&	0.0629	&	147.946	&		&		\\
590380	&	16981372	&	65825	&	20528516	&	1.21	&		&	0.0439	&	324.469	&		&		\\
649570	&	18703960	&	72415	&	22626536	&	1.21	&		&	0.0493	&	448.260	&		&		
\end{tabular}
\end{footnotesize}
\caption{Comparison of standard direct solve and null space method for Naiver-Stokes.}
\label{tab:ns}
\end{table}

\section{Tables for Dual-Mixed Stokes}\label{sec:Adual}
\begin{table}[h]\centering
\begin{footnotesize}
\begin{tabular}{r|r|r|r|r|r|r|r|r|r}
$n+1$	&	$|M|$	&	$|B|$	&	$|Z^TAZ|$	&	infl	&	diff	&	Ztime	&	NStime	&	Stime	&	speedup	\\\hline
5081	&	66946	&	2045	&	68764	&	1.03	&	1.48E-14	&	0.0017	&	0.108	&	0.124	&	1.15	\\
20161	&	267284	&	8078	&	274932	&	1.03	&	1.72E-14	&	0.0092	&	0.430	&	1.542	&	3.58	\\
45241	&	601830	&	18067	&	619310	&	1.03	&	3.05E-14	&	0.0098	&	1.186	&	8.180	&	6.90	\\
80321	&	1069456	&	32022	&	1101006	&	1.03	&	8.01E-14	&	0.0169	&	2.676	&	29.369	&	10.98	\\
125401	&	1671120	&	49850	&	1720576	&	1.03	&	1.17E-13	&	0.0203	&	5.600	&	60.643	&	10.83	\\
180481	&	2403874	&	70517	&	2476998	&	1.03	&	1.24E-13	&	0.0302	&	10.759	&	134.263	&	12.48	\\
245561	&	3274472	&	97532	&	3371716	&	1.03	&	2.84E-13	&	0.0452	&	16.867	&	213.088	&	12.63	\\
320641	&	4278110	&	126891	&	4406728	&	1.03	&	4.89E-13	&	0.0593	&	25.349	&	488.083	&	19.25	\\
405721	&	5410736	&	161950	&	5570722	&	1.03	&	2.45E-13	&	0.0846	&	44.851	&	818.019	&	18.24	\\
500801	&	6691372	&	199948	&	6890864	&	1.03	&	5.73E-13	&	0.1006	&	55.618	&	1138.317	&	20.47	\\
605881	&	8088772	&	242266	&	8326976	&	1.03	&	3.75E-13	&	0.1305	&	83.160	&	1657.527	&	19.93	\\
720961	&	9622302	&	284017	&	9913640	&	1.03	&	7.89E-13	&	0.1520	&	105.158	&	2926.810	&	27.83			
\end{tabular}
\end{footnotesize}
\caption{Comparison of standard direct solve and null space method for dual-mixed Stokes on a square domain.}
\label{tab:dmstokes}
\end{table}

\begin{table}[h]\centering
\begin{footnotesize}
\begin{tabular}{r|r|r|r|r|r|r|r|r|r}
$n+1$	&	$|M|$	&	$|B|$	&	$|Z^TAZ|$	&	infl	&	diff	&	Ztime	&	NStime	&	Stime	&	speedup	\\\hline
11359	&	148566	&	4668	&	152454	&	1.03	&	4.01E-13	&	0.0021	&	0.156	&	0.299	&	1.91	\\
26265	&	346440	&	10680	&	356004	&	1.03	&	1.95E-12	&	0.0057	&	0.517	&	2.141	&	4.14	\\
46421	&	614610	&	18767	&	631998	&	1.03	&	7.36E-12	&	0.0101	&	0.854	&	6.046	&	7.08	\\
72727	&	965192	&	29323	&	992846	&	1.03	&	1.80E-11	&	0.0129	&	1.402	&	14.981	&	10.68	\\
106360	&	1413945	&	42803	&	1454811	&	1.03	&	2.17E-11	&	0.0189	&	2.617	&	42.199	&	16.12	\\
139616	&	1857691	&	56146	&	1911593	&	1.03	&	5.05E-11	&	0.0234	&	3.889	&	66.351	&	17.06	\\
186072	&	2478247	&	74740	&	2550511	&	1.03	&	6.20E-11	&	0.0334	&	6.124	&	143.892	&	23.50	\\
234428	&	3124189	&	94049	&	3215679	&	1.03	&	8.54E-11	&	0.0458	&	8.915	&	248.958	&	27.93	\\
288984	&	3853383	&	115940	&	3966367	&	1.03	&	1.34E-10	&	0.0528	&	10.886	&	303.924	&	27.92	\\
359490	&	4796221	&	144148	&	4937215	&	1.03	&	3.78E-10	&	0.0694	&	15.655	&	430.508	&	27.50	\\
411446	&	5490449	&	164916	&	5652063	&	1.03	&	3.41E-10	&	0.0844	&	24.274	&	635.733	&	26.19	\\
490102	&	6542539	&	196409	&	6735371	&	1.03	&	4.32E-10	&	0.0979	&	24.919	&	1099.457	&	44.12	\\
567858	&	7582451	&	227499	&	7806219	&	1.03	&	4.86E-10	&	0.1194	&	27.873	&	1281.732	&	45.99	\\
656314	&	8765865	&	262910	&	9024771	&	1.03	&	5.87E-10	&	0.1399	&	38.662	&	1604.043	&	41.49	
\end{tabular}
\end{footnotesize}
\caption{Comparison of standard direct solve and null space method for dual-mixed Stokes on an irregular domain.}
\label{tab:dmstokesa}
\end{table}

\bibliographystyle{siam}

\bibliography{../../refs}

\end{document}